\documentclass[aps,pre,onecolumn,superscriptaddress,amsmath,amssymb,10pt,showkeys]{revtex4-2}
\usepackage{amsthm}
\usepackage{amsfonts}
\usepackage{mathtools}
\usepackage{enumitem}
\usepackage{mdframed}
\usepackage{iftex}
\ifptex
  \PassOptionsToPackage{dvipdfmx}{hyperref}
\fi
\usepackage[hidelinks]{hyperref}
\AddToHook{env/thebibliography/begin}{\raggedright}

\newtheorem{theorem}{Theorem}[section]
\newtheorem{lemma}[theorem]{Lemma}
\newtheorem{proposition}[theorem]{Proposition}
\newtheorem{remark}[theorem]{Remark}

\newtheorem{example}[theorem]{Example}

\DeclareMathOperator{\dist}{dist}

\makeatletter
\renewenvironment{proof}[1][\proofname]{\par
  \pushQED{\qed}%
  \normalfont \topsep6\p@\@plus6\p@\relax
  \trivlist
  \item\relax
  {\bfseries
  #1\@addpunct{.}}\hspace\labelsep\ignorespaces
}{%
  \popQED\endtrivlist\@endpefalse
}
\makeatother

\begin{document}
\title{Gaussian Fluctuations of the Gradient of the Fr\'echet Function and Nonstandard Asymptotics of Sample Fr\'echet Means}
\author{Yutaro Nakabayashi}
\affiliation{Graduate School of Mathematical Sciences,
The University of Tokyo,
3-8-1 Komaba, Meguro-ku,
Tokyo 153-8914, Japan}
\date{October 20, 2026}
\begin{abstract}
We establish a central limit theorem for the gradient of the
Fr\'echet function in local Euclidean coordinates
on a metric space.
Assuming a unique population mean, a consistent measurable
sequence of sample means, semiconcavity or semiconvexity and
moment bounds for the squared distance functions, and local $C^{1,\alpha}$
regularity of the population Fr\'echet function, we prove
an asymptotic linear representation for its gradient at the sample mean.
The leading term is the negative empirical average of the
gradients of the squared distance functions at the population
mean; these gradients exist almost surely under our assumptions.
On complete connected Riemannian manifolds, uniqueness of
the population mean, existence of a $(2+\eta)$-moment, and local
$C^{1,\alpha}$ regularity suffice, and the leading term
is expressed through the logarithm map.
We also give a sufficient condition for the regularity of the 
population Fr\'echet function through probability bounds for
neighborhoods of cut loci.
We then show how the local behavior of the population Fr\'echet function transforms 
these Gaussian fluctuations into limit distributions for suitably rescaled sample Fr\'echet means, 
allowing for anisotropic scaling and non-Gaussian limits. 
A flat-torus example exhibits anisotropic scaling, and two examples on the
circle illustrate logarithmic smeariness and the failure of any
deterministic linear scaling to yield a nondegenerate limit,
respectively.
\end{abstract}
\maketitle
\section{Introduction}

A population Fr\'echet mean extends the expectation of a Euclidean random variable to a metric-space-valued random variable by minimizing expected squared distance; 
replacing the expectation by an empirical average gives a sample Fr\'echet mean~\cite{Frechet1948}. 
The geometry of the underlying space can lead to nonuniqueness of Fr\'echet means; see Remark~\ref{rem:basic-assumptions}.
When the population Fr\'echet mean is unique, a basic question is whether sample Fr\'echet means converge to it. 
An early strong law of large numbers in quasi-metric spaces is due to Ziezold~\cite{Ziezold1977}. 
On manifolds, Bhattacharya and Patrangenaru established consistency of intrinsic and extrinsic sample Fr\'echet means~\cite{BhattacharyaPatrangenaru2003}. 
A natural next question is whether this law of large numbers is accompanied by a central limit theorem.
In Part II, Bhattacharya and Patrangenaru established an early general CLT for intrinsic sample Fr\'echet means on Riemannian manifolds and extended their results for extrinsic means~\cite{BhattacharyaPatrangenaru2003,BhattacharyaPatrangenaru2005}.
If the Hessian of the Fr\'echet function at its minimum is nonsingular and the regularity assumptions hold, their intrinsic CLT gives a Gaussian limit at the usual $\sqrt{n}$ scale in local coordinates.
Kendall and Le extended the theory to independent random variables that need not be identically distributed~\cite{KendallLe2011}.
Bhattacharya and Lin obtained CLTs in more general metric spaces and showed that, on Riemannian manifolds, the support may meet the cut locus of the population mean if the mass of the cut loci of shrinking neighborhoods decays sufficiently fast~\cite{BhattacharyaLin2017}.
Eltzner, Galaz-Garc\'ia, Huckemann, and Tuschmann clarified the role of cut-locus stability in the CLT on closed Riemannian manifolds~\cite{EltznerGalazGarciaHuckemannTuschmann2021}.
Hotz, Le, and Wood proved a $\sqrt{n}$-CLT on compact Riemannian manifolds under conditions that allow the support to meet a cut locus of codimension one or two, and identified an additional term in the Hessian of the Fr\'echet function in codimension one~\cite{HotzLeWood2024}.

Even when the Fr\'echet function is smooth near its minimum, its Hessian there can be singular.
On the circle, Hotz and Huckemann~\cite{HotzHuckemann2015} showed that,
for a density continuous near the antipode of the mean, the quadratic term
vanishes when the density at the antipode equals $1/(2\pi)$.
How the density approaches this value determines the convergence rate in their examples.
Eltzner and Huckemann~\cite{EltznerHuckemann2019} called such slower-than-$n^{-1/2}$ behavior \emph{smeariness}.
They derived limit laws from a local power expansion of the Fr\'echet function
and constructed examples on spheres in which the sample mean has an $n^{-1/6}$ rate and a non-Gaussian limit.
Tran~\cite{Tran2021} computed higher derivatives of the Fr\'echet function
for absolutely continuous, rotationally symmetric laws on spheres
and obtained smeary examples in dimensions at least four.
Eltzner~\cite{Eltzner2022} then showed that smeariness can occur on spheres
even when the support stays away from the cut locus of the mean.

Hundrieser, Eltzner, and Huckemann~\cite{HundrieserEltznerHuckemann2024}
showed that the mean squared distance from the sample mean to the population mean
can be much larger than the Fr\'echet variance divided by $n$, even when the usual
$n^{-1/2}$ rate holds asymptotically.
They termed this phenomenon \emph{finite sample smeariness}.
In later work~\cite{HundrieserEltznerHuckemann2026}, they proved minimax lower
bounds for estimation near laws with nonunique Fr\'echet means.
They also constructed circular examples in which the sample mean converges
at a logarithmic rate and has a two-point limit under that normalization.

Eltzner and Huckemann~\cite[Theorem~2.11]{EltznerHuckemann2019} developed
a central limit theory for smeary Fr\'echet means under a restrictive
assumption on the local power expansion. In rotated coordinates,
the leading part of the expansion is a sum of single-coordinate powers
with a common exponent. Hotz, Le, and Wood~\cite[Section~2.2]{HotzLeWood2024}
suggested that a rank-deficient Hessian could give rise to different
convergence rates in different directions.

The results above use different descriptions of the population Fr\'echet
function near its minimum. They rely on whether the Hessian exists and
is nonsingular, or on the form of the higher-order expansion when it
is singular. Which aspects of this local behavior determine the scaling
of the sample mean, and how do they enter its limiting law?

We establish an asymptotic linear representation for the gradient
of the population Fr\'echet function evaluated at a sample Fr\'echet mean.
Let $\mu$ be the unique population Fr\'echet mean and let
$\phi$ be a coordinate homeomorphism near $\mu$ with $\phi(\mu)=0$.
Write
\begin{equation*}
\rho_q(x)=\frac12d\left(\phi^{-1}(x),q\right)^2,
\qquad
f(x)=F\left(\phi^{-1}(x)\right).
\end{equation*}
For a consistent measurable sequence of sample Fr\'echet means,
let $x_n$ denote the coordinates of the $n$th sample mean, set equal
to $0$ when it lies outside a fixed coordinate neighborhood.

We assume that the functions $\rho_q$ are either all semiconcave
or all semiconvex on a common coordinate neighborhood, with integrable
semiconcavity or semiconvexity constants and Lipschitz constants in
$L^{2+\eta}(P)$ for some $\eta>0$. We also assume that $f$ is locally
$C^{1,\alpha}$ for some $\alpha\in(0,1]$.
Under these assumptions, $\rho_X$ is almost surely differentiable at $0$, and
\begin{equation*}
\sqrt{n}\nabla f(x_n)
=-\frac1{\sqrt{n}}\sum_{i=1}^n\nabla\rho_{X_i}(0)
+o_{\mathbb{P}}(1).
\end{equation*}
Consequently, $\sqrt{n}\nabla f(x_n)$ converges in distribution to a
centered Gaussian vector with covariance
$\operatorname{Cov}\left(\nabla\rho_X(0)\right)$.
The representation also applies when the Hessian at the minimum is singular.

On a complete connected Riemannian manifold, a $(2+\eta)$-moment
for some $\eta>0$ gives the semiconcavity and Lipschitz bounds in
normal coordinates. If the population mean is unique and the population
Fr\'echet function is locally $C^{1,\alpha}$ for some
$\alpha\in(0,1]$, the expansion becomes
\begin{equation*}
\sqrt{n}\nabla f(x_n)
=\frac1{\sqrt{n}}\sum_{i=1}^n\log_\mu(X_i)
+o_{\mathbb{P}}(1),
\end{equation*}
where $f=F\circ \exp_\mu$.
If $f$ has a nonsingular Hessian at $0$,
the classical $\sqrt{n}$-CLT for $x_n$ follows;
see Remark~\ref{rem:classical-clt}.
Under the moment assumption above,
Proposition~\ref{prop:frechet-c1alpha-cut-neighborhood} gives local
$C^{1,\alpha}$ regularity whenever
$\mathbb{P}(\dist(p,\operatorname{Cut}(X))\leq t)=O(t^\beta)$
uniformly for $p$ near $\mu$, for any $\beta\in(0,1]$.

Let $\zeta$ denote the Gaussian limit of $\sqrt{n}\nabla f(x_n)$.
To find a limit for $A_nx_n$, where $A_n$ is a linear isomorphism,
first suppose that the maps $y\mapsto\sqrt{n}\nabla f(A_n^{-1}y)$
converge locally uniformly to a homeomorphism $\Lambda$.
At $y=A_nx_n$, the rescaled gradient is exactly $\sqrt{n}\nabla f(x_n)$,
which suggests the approximation
\begin{equation*}
\sqrt{n}\nabla f(x_n)
=\sqrt{n}\nabla f\left(A_n^{-1}(A_nx_n)\right)
\approx\Lambda(A_nx_n).
\end{equation*}
This suggests applying $\Lambda^{-1}$ to the Gaussian limit.
Theorem~\ref{thm:gradient-scaling-limit} makes this argument rigorous
and gives $A_nx_n\xrightarrow{d}\Lambda^{-1}(\zeta)$.

If $\nabla f$ has a local inverse near $0$, then, with probability
tending to one, we can recover $A_nx_n$ directly from the gradient.
\begin{equation*}
A_nx_n=A_n(\nabla f)^{-1}
\left(\frac{\sqrt{n}\nabla f(x_n)}{\sqrt{n}}\right).
\end{equation*}
Thus, if the maps $y\mapsto A_n(\nabla f)^{-1}(y/\sqrt{n})$
converge locally uniformly to $\Lambda$, we expect $A_nx_n$ to
converge in distribution to $\Lambda(\zeta)$.
Theorem~\ref{thm:inverse-gradient-scaling-limit} establishes this
conclusion even when the convergence holds only on an open set
of full Gaussian probability.

In Example~\ref{ex:hexagonal-anisotropic-torus}, we construct a law
on a flat torus whose Fr\'echet function has a rank-one Hessian
at its minimum. This realizes the possibility raised by Hotz, Le, and
Wood~\cite[Section~2.2]{HotzLeWood2024} that a rank-deficient Hessian
can lead to different convergence rates in different directions. Applying
Theorem~\ref{thm:gradient-scaling-limit}, we obtain the rate $n^{-1/2}$
in the direction corresponding to the positive eigenvalue and $n^{-1/4}$
in the null direction. A mixed term couples the two directions and yields
a non-Gaussian limit even for the coordinate with the $\sqrt{n}$
normalization.

Example~\ref{ex:logarithmic-rate} revisits logarithmic smeariness on
the circle~\cite[Example~2]{HundrieserEltznerHuckemann2026}.
The sample mean converges at a logarithmic rate, with a two-point limit
under the corresponding normalization. The rescaled gradients have
no finite locally uniform limit at this scale. The rescaled inverse
gradients converge to the sign function uniformly on compact subsets
of $\mathbb{R}\setminus\{0\}$, so
Theorem~\ref{thm:inverse-gradient-scaling-limit} gives the limit.

In Example~\ref{ex:oscillatory-no-limit}, the gradient of a locally
$C^{1,1}$ population Fr\'echet function on the circle oscillates
under rescaling. Although $\sqrt{n}\nabla f(x_n)$ has a Gaussian
limit and $\sqrt{n}x_n$ is tight, no deterministic linear
scaling of $x_n$ yields a nondegenerate limit.

The paper is organized as follows.
Section~\ref{sec:setting-main-result} introduces the metric-space
framework and states the main theorems.
Section~\ref{sec:directional-representation} uses one-sided derivatives
to bound $\nabla f(x_n)$ in terms of the empirical process.
Section~\ref{sec:empirical-process-bounds} proves the bracketing
and asymptotic equicontinuity estimates and completes the proof
of Theorem~\ref{thm:main-gradient-clt}.
Section~\ref{sec:scaling-limits} proves the scaling theorems.
Section~\ref{sec:riemannian-manifolds} verifies the assumptions
on complete Riemannian manifolds and gives a sufficient
condition for the regularity of the population Fr\'echet function.
Section~\ref{sec:examples} presents the anisotropic, logarithmic,
and oscillatory examples.

\section{Framework and main results}
\label{sec:setting-main-result}
Let $(Q,d)$ be a metric space equipped with its Borel
$\sigma$-algebra.
Let $X$ be a $Q$-valued random variable defined on a probability
space $(\Omega,\mathcal{F},\mathbb{P})$, and let
$X_1,X_2,\ldots$ be independent copies of $X$.
Their common distribution on $Q$ is denoted by $P$, and
$\mathbb{E}$ denotes expectation with respect to $\mathbb{P}$.
Assume that, for some $o\in Q$,
\begin{equation*}
\mathbb{E}\left[d(o,X)^2\right]<\infty.
\end{equation*}
By the triangle inequality, this condition then holds with
any other base point in place of $o$.

For each $q\in Q$, call $p\mapsto d(p,q)^2/2$ the squared
distance function centered at $q$.
Its expectation and empirical average are the population
and empirical Fr\'echet functions, respectively:
\begin{equation*}
F(p)\coloneqq\frac12\mathbb{E}\left[d(p,X)^2\right],
\qquad
F_n(p)\coloneqq\frac{1}{2n}\sum_{i=1}^n d(p,X_i)^2,
\qquad p\in Q.
\end{equation*}
A minimizer of $F$ is called a population Fr\'echet mean,
and a minimizer of $F_n$ is called a sample Fr\'echet mean.

Let
\begin{equation*}
P_n\coloneqq\frac1n\sum_{i=1}^n\delta_{X_i}
\end{equation*}
be the empirical measure.
For an integrable scalar- or vector-valued function $h$,
we write
\begin{equation*}
Ph\coloneqq\int_Q h(q)\,P(dq)
=\mathbb{E}\left[h(X)\right],
\qquad
P_nh\coloneqq\frac1n\sum_{i=1}^n h(X_i),
\end{equation*}
and set
\begin{equation*}
\left\|h\right\|_{P,2}
\coloneqq
\left(\mathbb{E}\left[|h(X)|^2\right]\right)^{1/2}.
\end{equation*}
Here and below, $|\cdot|$ denotes the Euclidean norm and we write $B_r(x)$ for the open Euclidean ball of radius $r$
centered at $x$.

We impose the following assumptions.

\begin{enumerate}[
  label=\textup{(A\arabic*)},
  ref=A\arabic*,
  leftmargin=*
]

\item\label{ass:local-coordinates-consistency}
\textbf{Consistency of sample Fr\'echet means.}
The function $F$ has a unique minimizer $\mu\in Q$.
There exists a sequence $\{\mu_n\}_{n\geq1}$ of $Q$-valued
random variables satisfying
\begin{equation*}
\mu_n\in\operatorname*{argmin}_{p\in Q}F_n(p)
\qquad\text{a.s.},\ n\ge1.
\end{equation*}
We assume that
\begin{equation*}
d(\mu_n,\mu)\xrightarrow{\mathbb{P}}0.
\end{equation*}

\item\label{ass:locally-Euclidean}
\textbf{Local Euclidean structure near $\mu$.}
There exist an integer $m\geq1$, a radius $r>0$,
an open neighborhood $U\subset Q$ of $\mu$, and
a homeomorphism
\begin{equation*}
\phi\colon U\longrightarrow B_{r}(0)\subset\mathbb{R}^m
\end{equation*}
such that $\phi(\mu)=0$.
For $q\in Q$ and $x\in B_{r}(0)$, define
\begin{equation*}
\rho_q(x)\coloneqq\frac12d\left(\phi^{-1}(x),q\right)^2,
\end{equation*}
and set
\begin{equation*}
f(x)\coloneqq F\left(\phi^{-1}(x)\right)
=\mathbb{E}\left[\rho_X(x)\right],
\qquad
f_n(x)\coloneqq F_n\left(\phi^{-1}(x)\right)
=\frac1n\sum_{i=1}^n\rho_{X_i}(x).
\end{equation*}

\item\label{ass:semiconcavity-moments}
\textbf{Semiconcavity or semiconvexity with moment bounds.}
There exist $\sigma\in\{-1,1\}$ and Borel measurable functions
$c,L\colon Q\to[0,\infty)$ such that, for every $q\in Q$,
the function
\begin{equation*}
x\longmapsto \sigma\rho_q(x)-\frac{c(q)}2|x|^2
\end{equation*}
is concave on $B_r(0)$, and
\begin{equation*}
\left|\rho_q(x)-\rho_q(y)\right|
\leq L(q)|x-y|,
\qquad x,y\in B_{r}(0).
\end{equation*}
Moreover, for some $\eta>0$,
\begin{equation*}
\mathbb{E}\left[c(X)\right]<\infty,
\qquad
\mathbb{E}\left[L(X)^{2+\eta}\right]<\infty.
\end{equation*}

\item\label{ass:population-regularity}
\textbf{Regularity of the population Fr\'echet function.}
For some $\alpha\in(0,1]$,
\begin{equation*}
f\in C^{1,\alpha}\left(B_{r}(0)\right).
\end{equation*}

\end{enumerate}

\begin{remark}\label{rem:basic-assumptions}
\begin{enumerate}[label=\textup{(\roman*)},leftmargin=*]
\item
Existence and uniqueness of Fr\'echet means are not automatic
on general metric spaces, even under the second-moment
assumption.
For example, let $Q=\mathbb{R}\setminus\{0\}$ with the
Euclidean distance and
$P=(\delta_{-1}+\delta_1)/2$.
Then
\begin{equation*}
F(p)=\frac12(p^2+1),\qquad p\in Q,
\end{equation*}
so the infimum of $F$ is not attained.
Uniqueness may fail even on compact spaces: the uniform
distribution on the circle with its geodesic distance
has a constant Fr\'echet function.
In \textup{(A1)}, uniqueness is required only for the
population mean; sample means may be nonunique.

\item
If $(Q,d)$ is proper, meaning that every closed bounded
subset is compact, the second-moment assumption ensures
that $F$ and every $F_n$ have nonempty compact minimizer
sets.
If the population mean $\mu$ is unique, every measurable
sequence of sample Fr\'echet means satisfies
\begin{equation*}
d(\mu_n,\mu)\longrightarrow0
\qquad\text{almost surely};
\end{equation*}
see \cite{BhattacharyaPatrangenaru2003}.
Hence, in this setting, the consistency requirement in
\textup{(A1)} follows from the other assumptions.

\item
Assumption \textup{(A2)} imposes a substantial restriction
on the local topology at $\mu$.
In particular, an open book with at least three pages
is not locally Euclidean at points of its spine.
Thus the present framework does not directly cover
sticky means on the spine of an open book, as studied
in \cite{HotzEtAl2013}.
The restriction is local to $\mu$ and allows singularities
elsewhere in $Q$.

\item
Assumption \textup{(A3)} allows the squared distance functions
to be semiconcave \textup{(}$\sigma=1$\textup{)} or semiconvex \textup{(}$\sigma=-1$\textup{)}
in the coordinates of \textup{(A2)}.
On a complete connected Riemannian manifold, normal coordinates
$\phi=\exp_\mu^{-1}$ give the semiconcave case;
under a $(2+\eta)$-moment condition, \textup{(A3)} follows
from Lemma~\ref{lem:uniform-semiconcavity}.
For $Q=(\mathbb{R}^m,\|\cdot\|_p)$ with $1\leq p\leq\infty$,
the natural identification with Euclidean coordinates centered
at $\mu$ makes each squared distance function convex, hence
semiconvex.
Under a $(2+\eta)$-moment condition, \textup{(A3)} is satisfied.
A related example is $\mathbb{R}$ with the metric
$d(a,b)=|a-b|^{1/2}$.
Here $\rho_q(x)=|\mu+x-q|/2$ is convex, and sample
Fr\'echet means are sample medians.
With a unique population median, a finite first moment, and suitable
local H\"older regularity of the distribution function, our results recover
classical limit theorems for sample medians, including
nonstandard scalings;
see \cite[Problem~5.13]{Knight2000}.

\item
For complete connected Riemannian manifolds,
Proposition~\ref{prop:frechet-c1alpha-cut-neighborhood}
provides a sufficient condition for \textup{(A4)}.
Under the moment condition stated there, a locally uniform
bound on the probability that a base point lies close to
the cut locus of $X$ implies the required local
$C^{1,\alpha}$ regularity of the population Fr\'echet function.

\end{enumerate}
\end{remark}

Throughout the paper, $\nabla f$ denotes the gradient
with respect to the standard Euclidean inner product
on $\mathbb{R}^m$.
The notation $o_{\mathbb{P}}(1)$ and $O_{\mathbb{P}}(1)$
refers to convergence to zero in probability and
boundedness in probability, respectively.

Using the radius $r$ in \textup{(\ref{ass:locally-Euclidean})}, define
\begin{equation*}
x_n\coloneqq
\begin{cases}
\phi(\mu_n),
& \mu_n\in\phi^{-1}\left(B_{r/3}(0)\right),\\
0,
& \mu_n\notin\phi^{-1}\left(B_{r/3}(0)\right).
\end{cases}
\end{equation*}
Assumptions
\textup{(\ref{ass:local-coordinates-consistency})}
and \textup{(\ref{ass:locally-Euclidean})}
imply that
\begin{equation*}
\mathbb{P}\left(
\mu_n\in\phi^{-1}\left(B_{r/3}(0)\right)
\right)\longrightarrow1,
\qquad
x_n\xrightarrow{\mathbb{P}}0.
\end{equation*}

We now state the main result.

\begin{theorem}
\label{thm:main-gradient-clt}
Assume
\textup{(\ref{ass:local-coordinates-consistency})}--%
\textup{(\ref{ass:population-regularity})}.
Then $\rho_q$ is differentiable at $0$ for $P$-almost
every $q\in Q$, and
\begin{equation*}
\mathbb{E}\left[\nabla\rho_X(0)\right]=0,
\qquad
\mathbb{E}\left[|\nabla \rho_X(0)|^2\right]<\infty.
\end{equation*}
Moreover,
\begin{equation}\label{eq:main-asymptotic-expansion}
\sqrt{n}\nabla f(x_n)
=
-\frac1{\sqrt{n}}\sum_{i=1}^n\nabla\rho_{X_i}(0)
+
o_{\mathbb{P}}(1)
\end{equation}
in $\mathbb{R}^m$.
Consequently,
\begin{equation}\label{eq:main-gradient-clt}
\sqrt{n}\nabla f(x_n)
\xrightarrow{d}
\mathcal{N}_{\mathbb{R}^m}(0,\Sigma),
\qquad
\Sigma\coloneqq\operatorname{Cov}\left(\nabla \rho_X(0)\right).
\end{equation}
\end{theorem}

The following theorem gives a limit for the sample Fr\'echet mean
when $\nabla f$ admits
a locally uniform scaling limit.

\begin{theorem}
\label{thm:gradient-scaling-limit}
Assume \textup{(\ref{ass:local-coordinates-consistency})}--%
\textup{(\ref{ass:population-regularity})}.
Let $A_n$ be deterministic linear isomorphisms of
$\mathbb{R}^m$ satisfying $\|A_n^{-1}\|_{\mathrm{op}}\to0$.
Suppose that there exists a homeomorphism $\Lambda$ of
$\mathbb{R}^m$ such that, for every compact set $K\subset\mathbb{R}^m$,
\begin{equation}\label{eq:gradient-scaling-limit}
\sup_{y\in K}
\left|
\sqrt{n}\nabla f\left(A_n^{-1}y\right)-\Lambda(y)
\right|
\longrightarrow0.
\end{equation}
The expression on the left is well defined for all sufficiently
large $n$, since $A_n^{-1}K\subset B_{r}(0)$.
Then
\begin{equation}\label{eq:scaled-frechet-mean-limit}
A_nx_n\xrightarrow{d}\Lambda^{-1}(\zeta),
\qquad
\zeta\sim\mathcal{N}_{\mathbb{R}^m}(0,\Sigma),
\end{equation}
where $\Sigma$ is the covariance matrix in
Theorem~\ref{thm:main-gradient-clt}.
\end{theorem}

When $\nabla f$ is locally invertible, one can instead
rescale its inverse. The following theorem permits convergence on an
open set of full Gaussian probability and applies to the logarithmic
example in Section~\ref{subsec:logarithmic-smeariness}.

\begin{theorem}
\label{thm:inverse-gradient-scaling-limit}
Assume \textup{(\ref{ass:local-coordinates-consistency})}--%
\textup{(\ref{ass:population-regularity})}.
Let $\zeta\sim\mathcal{N}_{\mathbb{R}^m}(0,\Sigma)$, where $\Sigma$ is
the covariance matrix in Theorem~\ref{thm:main-gradient-clt}.
Suppose that there exists an open neighborhood
$V\subset B_r(0)$ of $0$ such that
\begin{equation*}
\nabla f|_V\colon V\longrightarrow\nabla f(V)
\end{equation*}
is a homeomorphism.
Let $A_n$ be deterministic linear isomorphisms of $\mathbb{R}^m$.
Suppose that there exist an open set $D\subset\mathbb{R}^m$ with
$\mathbb{P}(\zeta\in D)=1$ and a continuous map
$\Lambda\colon D\to\mathbb{R}^m$ such that, for every compact set
$K\subset D$,
\begin{equation}\label{eq:inverse-gradient-scaling-limit}
\sup_{y\in K}
\left|A_n(\nabla f|_V)^{-1}\left(\frac{y}{\sqrt n}\right)-\Lambda(y)\right|
\longrightarrow0.
\end{equation}
The expression is well defined for all sufficiently large $n$,
since $\nabla f(V)$ is an open neighborhood of $0$.
Then
\begin{equation}\label{eq:inverse-scaled-frechet-mean-limit}
A_nx_n\xrightarrow{d}\Lambda(\zeta).
\end{equation}
\end{theorem}

Here $D$ need not contain $0$, and $\Lambda$ need not be injective
or extend continuously to all of $\mathbb{R}^m$.
In the logarithmic example, one can take
$D=\mathbb{R}\setminus\{0\}$ and $\Lambda(y)=\operatorname{sgn}(y)$.

Theorem~\ref{thm:gradient-scaling-limit} includes the classical central
limit theorem for sample Fr\'echet means as a special case.
\begin{remark}
\label{rem:classical-clt}
Assume \textup{(\ref{ass:local-coordinates-consistency})}--%
\textup{(\ref{ass:population-regularity})}.
Suppose further that $\nabla f$ is differentiable at $0$,
with nonsingular derivative $H$.
Then
\begin{equation*}
\sqrt{n}\,x_n
\xrightarrow{d}
\mathcal{N}_{\mathbb{R}^m}
\left(0,H^{-1}\Sigma(H^{-1})^{\mathsf T}\right).
\end{equation*}
Indeed, taking $A_n=\sqrt{n}\,I$, differentiability of
$\nabla f$ at $0$ gives
\begin{equation*}
\sqrt{n}\nabla f(A_n^{-1}y)\longrightarrow Hy
\end{equation*}
locally uniformly in $y$.
The conclusion follows from
Theorem~\ref{thm:gradient-scaling-limit}
with $\Lambda(y)=Hy$.
\end{remark}

\section{One-sided derivatives and the empirical process}
\label{sec:directional-representation}

Throughout this section and
Section~\ref{sec:empirical-process-bounds}, we assume
\textup{(\ref{ass:local-coordinates-consistency})}--%
\textup{(\ref{ass:population-regularity})}
and use the notation introduced in
Section~\ref{sec:setting-main-result}.
One-sided partial derivatives provide a jointly measurable
family defined even where squared distance
functions are not differentiable.

Set
\begin{equation*}
V\coloneqq B_{r/3}(0),\qquad K\coloneqq\overline{B_{2r/3}(0)}.
\end{equation*}
For $q\in Q$, define the convex function
\begin{equation}\label{eq:def-convexified-loss}
\psi_q(x)\coloneqq\frac{c(q)}2|x|^2-\sigma\rho_q(x),
\qquad x\in B_r(0).
\end{equation}
Its expectation is finite and convex, with
\begin{equation*}
\psi(x)\coloneqq P\psi_\cdot(x)
=\frac{Pc}{2}|x|^2-\sigma f(x),
\qquad
\nabla\psi(x)=(Pc)x-\sigma\nabla f(x).
\end{equation*}
In particular, $\psi\in C^{1,\alpha}(B_r(0))$.

For $x\in B_r(0)$ and $v\in\mathbb{R}^m$, define
\begin{equation*}
D\rho_q(x;v)\coloneqq
\lim_{t\downarrow0}
\frac{\rho_q(x+tv)-\rho_q(x)}{t}.
\end{equation*}
Assumption
\textup{(\ref{ass:semiconcavity-moments})}
implies that this finite limit exists, and
\begin{equation}\label{eq:directional-bound}
|D\rho_q(x;v)|\leq L(q)|v|.
\end{equation}
Indeed, the function
$t\mapsto\sigma\rho_q(x+tv)-c(q)|x+tv|^2/2$
is concave, so its right difference quotients are monotone.
The Lipschitz bound on $\rho_q$ gives finiteness
and \eqref{eq:directional-bound}.

For fixed $x,q$, the map
$v\mapsto\sigma D\rho_q(x;v)$ is concave,
and $v\mapsto D\rho_q(x;v)$ is $L(q)$-Lipschitz.
The function $\rho_q$ is differentiable at $x$
if and only if $v\mapsto D\rho_q(x;v)$ is linear.
The concavity assertion and the characterization
of differentiability follow by applying
\cite[Theorems~23.1 and~25.2]{Rockafellar1970}
to the convex function $\psi_q$.
The $L(q)$-Lipschitz property follows directly
from the Lipschitz bound on $\rho_q$.

Let $e_1,\ldots,e_m$ be the standard basis of $\mathbb{R}^m$.
For $x\in V$ and $j=1,\ldots,m$, define
\begin{equation}\label{eq:def-directional-field}
T_{x,j}^{+}(q)\coloneqq-D\rho_q(x;e_j),
\qquad
T_{x,j}^{-}(q)\coloneqq D\rho_q(x;-e_j).
\end{equation}
Set
\begin{equation*}
T_x^\pm(q)\coloneqq
(T_{x,1}^\pm(q),\ldots,T_{x,m}^\pm(q)).
\end{equation*}

\begin{lemma}\label{lem:directional-field}
The maps $(x,q)\mapsto T_x^\pm(q)$ are measurable with respect
to $\mathcal{B}(V)\otimes\mathcal{B}(Q)$, and
\begin{equation}\label{eq:directional-field-bound}
|T_{x,j}^\pm(q)|\leq L(q),\qquad
|T_x^\pm(q)|\leq\sqrt{m}\,L(q).
\end{equation}
For $x\in V$ and $s>0$ such that $x\pm se_j\in B_{r}(0)$,
\begin{equation}\label{eq:directional-secant-bounds}
\sigma\frac{\rho_q(x-se_j)-\rho_q(x)}s-\frac{c(q)s}{2}
\leq\sigma T_{x,j}^\pm(q)\leq
\sigma\frac{\rho_q(x)-\rho_q(x+se_j)}s+\frac{c(q)s}{2}.
\end{equation}
\end{lemma}

\begin{proof}
Since $V$ has a countable base, the Borel $\sigma$-algebra
of the product topology satisfies
\begin{equation*}
\mathcal{B}(V\times Q)=\mathcal{B}(V)\otimes\mathcal{B}(Q).
\end{equation*}
Thus the difference quotients, and hence their pointwise
limits, are $\mathcal{B}(V)\otimes\mathcal{B}(Q)$-measurable.
The bounds in \eqref{eq:directional-field-bound} follow
from \eqref{eq:directional-bound}.

For the convex function $\psi_q$, the left and right partial
derivatives along $e_j$ lie between its backward and forward
difference quotients:
\begin{equation*}
\frac{\psi_q(x)-\psi_q(x-se_j)}s
\leq-D\psi_q(x;-e_j)
\leq D\psi_q(x;e_j)\leq
\frac{\psi_q(x+se_j)-\psi_q(x)}s.
\end{equation*}
Since
\begin{equation*}
\begin{aligned}
-D\psi_q(x;-e_j)
&=c(q)\langle x,e_j\rangle+\sigma T_{x,j}^-(q),\\
D\psi_q(x;e_j)
&=c(q)\langle x,e_j\rangle+\sigma T_{x,j}^+(q),
\end{aligned}
\end{equation*}
substituting \eqref{eq:def-convexified-loss} proves
\eqref{eq:directional-secant-bounds}.
\end{proof}

The vectors $T_x^\pm(q)$ agree with $-\nabla\rho_q(x)$
where $\rho_q$ is differentiable. Otherwise, their coordinates
are the negatives of the right and left partial derivatives,
respectively.
The argument below uses only the coordinatewise bounds
and does not require differentiability of $\rho_q$ at $x$.

The next lemma is based on
\cite[Corollary~2.4]{LytchakHuckemann2025}.
We give a proof here to include the additional conclusions
that follow from our assumptions.

\begin{lemma}\label{lem:population-directional-identification}
For each fixed $x\in V$, the function $\rho_q$ is
differentiable at $x$ for $P$-almost every $q\in Q$, and
\begin{equation*}
T_x^\pm(q)=-\nabla\rho_q(x)
\qquad\text{for $P$-almost every $q\in Q$}.
\end{equation*}
Moreover,
\begin{equation}\label{eq:mean-directional-field}
PT_x^\pm=-\mathbb{E}\left[\nabla\rho_X(x)\right]=-\nabla f(x).
\end{equation}
\end{lemma}

\begin{proof}
Fix $x\in V$ and $v\in\mathbb{R}^m$.
The difference quotients of $\rho_q$ in direction $v$
are bounded by $L(q)|v|$, an integrable function.
Dominated convergence and differentiability of $f$ give
\begin{equation}\label{eq:directional-expectation}
P D\rho_\cdot(x;v)=\langle\nabla f(x),v\rangle.
\end{equation}
Taking $v=\pm e_j$ yields $PT_x^\pm=-\nabla f(x)$.

By concavity of $v\mapsto\sigma D\rho_q(x;v)$,
\begin{equation*}
\sigma\left(D\rho_q(x;v)+D\rho_q(x;-v)\right)\leq0.
\end{equation*}
Its expectation is zero by \eqref{eq:directional-expectation}.
Thus it vanishes for $P$-almost every $q$, for each $v$.
Taking a countable dense set of directions and using continuity
in $v$, we conclude that $D\rho_q(x;\cdot)$ is almost surely
odd. Hence $\sigma D\rho_q(x;\cdot)$ is both odd and concave,
so it is linear. Thus $\rho_q$ is almost surely differentiable
at $x$, and the asserted identities follow.
\end{proof}

The exceptional set in this lemma may depend on the
fixed point $x$. We do not require a common exceptional
set for all $x\in V$: the definitions in
\eqref{eq:def-directional-field} apply to every
$(x,q)\in V\times Q$.

Define
\begin{equation*}
E_n\coloneqq\{\mu_n\in\phi^{-1}(V)\}.
\end{equation*}
Then $\mathbb{P}(E_n)\to1$, $x_n\in V$ for every $n$,
and $x_n=\phi(\mu_n)$ on $E_n$.

\begin{lemma}\label{lem:exact-empirical-directional}
On $E_n$, for every $j=1,\ldots,m$,
\begin{equation}\label{eq:empirical-directional-optimality}
P_nT_{x_n,j}^+\leq0\leq P_nT_{x_n,j}^-.
\end{equation}
Consequently,
\begin{equation}\label{eq:gradient-sandwich}
(P_n-P)T_{x_n,j}^+
\leq\partial_jf(x_n)
\leq(P_n-P)T_{x_n,j}^-.
\end{equation}
\end{lemma}

\begin{proof}
On $E_n$, $f_n(x_n)\leq f_n(x)$ holds for every $x\in V$.
Thus, for each $j=1,\ldots,m$,
\begin{equation*}
Df_n(x_n;e_j)\geq0,
\qquad Df_n(x_n;-e_j)\geq0.
\end{equation*}
Since finite summation commutes with directional differentiation,
\begin{equation*}
P_nT_{x_n,j}^+=-Df_n(x_n;e_j)\leq0,
\qquad
P_nT_{x_n,j}^-=Df_n(x_n;-e_j)\geq0.
\end{equation*}
This proves \eqref{eq:empirical-directional-optimality}.
By \eqref{eq:mean-directional-field},
\begin{equation*}
PT_{x,j}^\pm=-\partial_jf(x),\qquad x\in V.
\end{equation*}
Evaluating these identities at $x_n$ and combining them with
\eqref{eq:empirical-directional-optimality} yields
\eqref{eq:gradient-sandwich}.
\end{proof}

For an integrable scalar- or vector-valued measurable function $h$,
define
\begin{equation*}
\mathbb{G}_nh\coloneqq\sqrt{n}(P_n-P)h.
\end{equation*}
Multiplying \eqref{eq:gradient-sandwich} by $\sqrt{n}$ gives,
on $E_n$,
\begin{equation}\label{eq:empirical-process-gradient}
\mathbb{G}_nT_{x_n,j}^+
\leq\sqrt{n}\,\partial_jf(x_n)
\leq\mathbb{G}_nT_{x_n,j}^-,
\qquad j=1,\ldots,m.
\end{equation}
Since $0$ minimizes $f$, Lemma~\ref{lem:population-directional-identification}
gives
\begin{equation*}
\mathbb{G}_nT_0^+=\mathbb{G}_nT_0^-
=-\frac1{\sqrt{n}}\sum_{i=1}^n
\nabla\rho_{X_i}(0)\qquad\text{almost surely}.
\end{equation*}
Thus the main asymptotic expansion follows once we prove,
for both signs,
\begin{equation}\label{eq:random-index-removal-target}
\mathbb{G}_nT_{x_n}^\pm-\mathbb{G}_nT_0^\pm
=o_{\mathbb{P}}(1).
\end{equation}
Indeed, subtracting the common term
$\mathbb{G}_nT_{0,j}^+=\mathbb{G}_nT_{0,j}^-$
from \eqref{eq:empirical-process-gradient} bounds
$\sqrt{n}\,\partial_jf(x_n)-\mathbb{G}_nT_{0,j}^+$
between two $o_{\mathbb{P}}(1)$ terms on $E_n$.
Since $\mathbb{P}(E_n)\to1$ and $m<\infty$, this yields
the desired expansion.
The Gaussian limit then follows from the multivariate
central limit theorem. We establish
\eqref{eq:random-index-removal-target} in the next section.

\section{Asymptotic equicontinuity of the empirical process}
\label{sec:empirical-process-bounds}

We retain the assumptions and notation of
Section~\ref{sec:directional-representation}.
In particular,
\begin{equation*}
V=B_{r/3}(0),
\qquad
K=\overline{B_{2r/3}(0)},
\end{equation*}
and
\begin{equation*}
T_{x,j}^+(q)=-D\rho_q(x;e_j),
\qquad
T_{x,j}^-(q)=D\rho_q(x;-e_j),
\end{equation*}
with
\begin{equation*}
\mathbb{G}_nh=\sqrt{n}(P_n-P)h.
\end{equation*}
Recall that, for both signs,
\begin{equation*}
|T_{x,j}^\pm(q)|\leq L(q),\qquad
|T_x^\pm(q)|\leq\sqrt{m}\,L(q),
\qquad x\in V,\quad q\in Q.
\end{equation*}

We now prove \eqref{eq:random-index-removal-target}.
First, we show that each coordinate class of $T_x^\pm$ has
polynomial bracketing numbers, which yield asymptotic equicontinuity.
Second, we prove that $T_x^\pm\to T_0^\pm$ in $L^2(P)$ as $x\to0$.
Asymptotic equicontinuity then allows us to replace $x_n$
by $0$ in both empirical processes.

\subsection{Bracketing entropy}
\label{subsec:bracketing-entropy}

For Borel measurable functions
$\ell,u\colon Q\to\mathbb{R}$ satisfying $\ell\leq u$,
define the bracket
\begin{equation*}
[\ell,u]
\coloneqq
\left\{
h\colon Q\to\mathbb{R}
\mid
\ell(q)\leq h(q)\leq u(q)
\text{ for every }q\in Q
\right\}.
\end{equation*}
Its $L^2(P)$-width is $\|u-\ell\|_{P,2}$.
For a class $\mathcal{F}$ of real-valued Borel measurable
functions, let
\begin{equation*}
N_{[]}\left(\varepsilon,\mathcal{F},L^2(P)\right)
\end{equation*}
denote the smallest number of brackets of $L^2(P)$-width
at most $\varepsilon$ required to cover $\mathcal{F}$.

The following asymptotic equicontinuity result follows from
the bracketing central limit theorem
\cite[Theorem~2.5.6]{VanDerVaartWellner1996}
and the characterization of $P$-Donsker classes in terms of
asymptotic equicontinuity
\cite[Sections~1.5 and~2.1]{VanDerVaartWellner1996}.
Probabilities involving possibly nonmeasurable suprema
are understood as outer probabilities, denoted by $\mathbb{P}^{*}$.

\begin{theorem}
\label{thm:bracketing-equicontinuity}
Let $\mathcal{F}$ be a class of real-valued Borel measurable
functions on $Q$ admitting an envelope $H\in L^2(P)$.
Suppose that
\begin{equation*}
\int_0^1
\sqrt{
\log N_{[]}\left(\varepsilon,\mathcal{F},L^2(P)\right)
}
\,d\varepsilon
<\infty.
\end{equation*}
Then the empirical process is asymptotically equicontinuous
in $L^2(P)$: for every $\varepsilon>0$,
\begin{equation}
\label{eq:asymptotic-equicontinuity}
\lim_{\delta\downarrow0}
\limsup_{n\to\infty}
\mathbb{P}^{*}
\left(
\sup_{\substack{
h_1,h_2\in\mathcal{F}\\
\|h_1-h_2\|_{P,2}\leq\delta
}}
\left|\mathbb{G}_n(h_1-h_2)\right|
>\varepsilon
\right)
=0.
\end{equation}
\end{theorem}

For each $j=1,\ldots,m$, set
\begin{equation*}
\mathcal{T}_j^\pm
\coloneqq
\left\{T_{x,j}^\pm\mid x\in V\right\}.
\end{equation*}

To apply Theorem~\ref{thm:bracketing-equicontinuity} to each
$\mathcal{T}_j^\pm$, we derive a polynomial bound on
$N_{[]}(\varepsilon,\mathcal{T}_j^\pm,L^2(P))$ for both signs.

\begin{lemma}
\label{lem:polynomial-bracketing}
For each $j=1,\ldots,m$, there exist constants
$A>0$ and $\beta>0$ such that, for both signs and all
sufficiently small $\varepsilon>0$,
\begin{equation}
\label{eq:polynomial-bracketing-number}
N_{[]}\left(\varepsilon,\mathcal{T}_j^\pm,L^2(P)\right)
\leq A\varepsilon^{-\beta}.
\end{equation}
Consequently, the empirical process indexed by each $\mathcal{T}_j^\pm$
is asymptotically equicontinuous in $L^2(P)$,
in the sense of \eqref{eq:asymptotic-equicontinuity}.
\end{lemma}

\begin{proof}
Fix $j\in\{1,\ldots,m\}$.
We first construct brackets covering the classes
$\sigma\mathcal{T}_j^\pm
\coloneqq\{\sigma h\mid h\in\mathcal{T}_j^\pm\}$.
If $\sigma=1$, these also cover $\mathcal{T}_j^\pm$;
if $\sigma=-1$, replacing each bracket $[\ell,u]$ by
$[-u,-\ell]$ gives a covering of $\mathcal{T}_j^\pm$
with unchanged bracket widths.

\medskip
\noindent\textbf{Step 1. Construction of the brackets.}
For sufficiently small $s>0$ and $0<\delta\leq s$,
we construct at most $C_{\mathrm{cover}}\delta^{-m}$
brackets covering $\sigma\mathcal{T}_j^+\cup\sigma\mathcal{T}_j^-$.
Here $C_{\mathrm{cover}}$ is independent of $s$ and $\delta$.

Since $V$ is a bounded subset of $\mathbb{R}^m$,
there exists $C_{\mathrm{cover}}>0$ such that,
for all sufficiently small $\delta>0$, a finite set
$\mathcal{A}_\delta\subset V$ satisfies
\begin{equation}
\label{eq:delta-net-cardinality}
V\subset
\bigcup_{a\in\mathcal{A}_\delta}B_\delta(a),
\qquad
\#\mathcal{A}_\delta
\leq C_{\mathrm{cover}}\delta^{-m}.
\end{equation}
Choose
\begin{equation}
\label{eq:difference-quotient-scale}
0<s<\min\left\{1,\frac r3\right\},
\qquad
0<\delta\leq s.
\end{equation}
For $x,a\in V$ with $|x-a|\leq\delta$,
the points $x\pm se_j$ and $a\pm se_j$ lie in $K$.

By \eqref{eq:directional-secant-bounds},
\begin{equation*}
\sigma\frac{\rho_q(x-se_j)-\rho_q(x)}s-\frac{c(q)s}{2}
\leq\sigma T_{x,j}^\pm(q)\leq
\sigma\frac{\rho_q(x)-\rho_q(x+se_j)}s+\frac{c(q)s}{2}.
\end{equation*}
By the Lipschitz bound in
\textup{(\ref{ass:semiconcavity-moments})},
replacing $x$ by $a$ in either quotient changes it
by at most $2L(q)\delta/s$.
Together with $|\sigma T_{x,j}^\pm(q)|\leq L(q)$, this suggests
the following endpoints:
\begin{equation*}
\begin{aligned}
\ell_{a,j}(q)
&\coloneqq
\max\left\{
-L(q),
\sigma\frac{\rho_q(a-se_j)-\rho_q(a)}{s}
-\frac{c(q)s}{2}
-\frac{2L(q)\delta}{s}
\right\},\\
u_{a,j}(q)
&\coloneqq
\min\left\{
L(q),
\sigma\frac{\rho_q(a)-\rho_q(a+se_j)}{s}
+\frac{c(q)s}{2}
+\frac{2L(q)\delta}{s}
\right\}.
\end{aligned}
\end{equation*}
These functions are Borel measurable and satisfy
\begin{equation*}
\ell_{a,j}(q)
\leq\sigma T_{x,j}^\pm(q)
\leq u_{a,j}(q),
\qquad q\in Q,
\end{equation*}
whenever $x\in V$ and $|x-a|\leq\delta$.
Moreover, $\ell_{a,j}\leq u_{a,j}$ and
\begin{equation*}
-L(q)\leq\ell_{a,j}(q)\leq u_{a,j}(q)\leq L(q).
\end{equation*}
Thus the endpoints belong to $L^2(P)$.
By \eqref{eq:delta-net-cardinality}, the brackets
$\{[\ell_{a,j},u_{a,j}]\}_{a\in \mathcal{A}_\delta}$
cover $\sigma\mathcal{T}_j^+\cup\sigma\mathcal{T}_j^-$.

\medskip
\noindent\textbf{Step 2. Bounds on bracket widths.}
Set
\begin{equation*}
M_\eta
\coloneqq
\mathbb{E}\left[L(X)^{2+\eta}\right]
<\infty.
\end{equation*}
We show that there exists $C_1>0$ such that,
for every $R\geq1$ and for any
$a\in\mathcal{A}_\delta$,
\begin{equation}
\label{eq:L2-bracket-width-bound}
\|u_{a,j}-\ell_{a,j}\|_{P,2}^2
\leq
2RC_1s^\alpha
+
8R\mathbb{E}\left[L(X)\right]\frac{\delta}{s}
+
4M_\eta R^{-\eta}.
\end{equation}

The bounds established in Step~1 give
\begin{equation}
\label{eq:bracket-width-uniform-bound}
0\leq u_{a,j}(q)-\ell_{a,j}(q)\leq2L(q).
\end{equation}
Moreover, the maximum in the lower endpoint and
the minimum in the upper endpoint can only decrease
the width. Hence
\begin{align*}
u_{a,j}(q)-\ell_{a,j}(q)
&\leq c(q)s-\sigma\frac{\rho_q(a+se_j)+\rho_q(a-se_j)-2\rho_q(a)}{s}+\frac{4L(q)\delta}{s}\\
&=
\frac{\psi_q(a+se_j)+\psi_q(a-se_j)-2\psi_q(a)}{s}+\frac{4L(q)\delta}{s}.
\end{align*}
Taking expectations yields
\begin{equation*}
\mathbb{E}\left[u_{a,j}(X)-\ell_{a,j}(X)\right]
\leq
\frac{
\psi(a+se_j)+\psi(a-se_j)-2\psi(a)
}{s}
+
4\mathbb{E}\left[L(X)\right]\frac{\delta}{s}.
\end{equation*}

Since
\begin{equation*}
\psi(x)=\frac{Pc}{2}|x|^2-\sigma f(x),
\end{equation*}
assumptions
\textup{(\ref{ass:semiconcavity-moments})}
and \textup{(\ref{ass:population-regularity})}
imply that $\psi\in C^{1,\alpha}(B_{r}(0))$.
As $K$ is compact, there exists $C_{\mathrm{H\ddot{o}l}}>0$
such that
\begin{equation*}
|\nabla\psi(u)-\nabla\psi(v)|
\leq C_{\mathrm{H\ddot{o}l}}|u-v|^\alpha,
\qquad u,v\in K.
\end{equation*}
Consequently,
\begin{align*}
\psi(a+se_j)+\psi(a-se_j)-2\psi(a)
&=\int_0^s
\left\langle
\nabla\psi(a+te_j)-\nabla\psi(a-te_j),e_j
\right\rangle\,dt\\
&\leq\frac{2^\alpha C_{\mathrm{H\ddot{o}l}}}{1+\alpha}
s^{1+\alpha}.
\end{align*}
Thus, for some $C_1>0$ independent of
$a$, $s$, and $\delta$,
\begin{equation*}
\mathbb{E}\left[u_{a,j}(X)-\ell_{a,j}(X)\right]
\leq
C_1s^\alpha
+
4\mathbb{E}\left[L(X)\right]\frac{\delta}{s}.
\end{equation*}

For every $R\geq1$, splitting according to whether
$L(X)\leq R$ and using
\eqref{eq:bracket-width-uniform-bound} gives
\begin{align*}
\|u_{a,j}-\ell_{a,j}\|_{P,2}^2
&\leq
2R\mathbb{E}\left[u_{a,j}(X)-\ell_{a,j}(X)\right]
+
4\mathbb{E}\left[
L(X)^2\mathbf{1}_{\{L(X)>R\}}
\right]\\
&\leq
2RC_1s^\alpha
+
8R\mathbb{E}\left[L(X)\right]\frac{\delta}{s}
+
4M_\eta R^{-\eta},
\end{align*}
where the last inequality uses
\begin{equation*}
\mathbb{E}\left[
L(X)^2\mathbf{1}_{\{L(X)>R\}}
\right]
\leq M_\eta R^{-\eta}.
\end{equation*}
This proves \eqref{eq:L2-bracket-width-bound}.

\medskip
\noindent\textbf{Step 3. Choice of parameters.}
We show that there exist constants $A>0$
and $\beta>0$ such that, for every sufficiently
small $\varepsilon>0$, the parameters $R$, $s$,
and $\delta$ can be chosen so that
\begin{equation*}
\|u_{a,j}-\ell_{a,j}\|_{P,2}
\leq\varepsilon,
\qquad a\in\mathcal{A}_\delta,
\end{equation*}
and
\begin{equation*}
\#\mathcal{A}_\delta
\leq A\varepsilon^{-\beta}.
\end{equation*}
These estimates imply
\eqref{eq:polynomial-bracketing-number} by Step~1.

Choose constants $c_R\geq1$ and $c_s,c_\delta>0$
such that
\begin{equation*}
4M_\eta c_R^{-\eta}\leq\frac13,
\qquad
2c_RC_1c_s^\alpha\leq\frac13,
\qquad
8c_R\mathbb{E}\left[L(X)\right]c_\delta
\leq\frac13.
\end{equation*}
For sufficiently small $\varepsilon>0$, set
\begin{equation*}
R\coloneqq c_R\varepsilon^{-2/\eta},
\qquad
s\coloneqq c_s\varepsilon^{(2+2/\eta)/\alpha},
\qquad
\delta\coloneqq
c_\delta\varepsilon^{2+2/\eta}s.
\end{equation*}
Then $R\geq1$ and
\eqref{eq:difference-quotient-scale} holds.
Each term on the right-hand side of
\eqref{eq:L2-bracket-width-bound}
is at most $\varepsilon^2/3$.
Therefore
\begin{equation*}
\|u_{a,j}-\ell_{a,j}\|_{P,2}
\leq\varepsilon,
\qquad a\in\mathcal{A}_\delta.
\end{equation*}

Since
\begin{equation*}
\delta
=
c_sc_\delta
\varepsilon^{(2+2/\eta)(1+1/\alpha)},
\end{equation*}
the cardinality estimate
\eqref{eq:delta-net-cardinality} gives
\begin{equation*}
\#\mathcal{A}_\delta
\leq
C_{\mathrm{cover}}(c_sc_\delta)^{-m}
\varepsilon^{-\beta},
\qquad
\beta\coloneqq
m\left(2+\frac{2}{\eta}\right)
\left(1+\frac1\alpha\right).
\end{equation*}
Taking
$A=C_{\mathrm{cover}}(c_sc_\delta)^{-m}$
proves \eqref{eq:polynomial-bracketing-number}.

Finally, $L\in L^2(P)$ is an envelope for
$\mathcal{T}_j^\pm$, and the bound on bracketing number
implies
\begin{equation*}
\int_0^1
\sqrt{
\log N_{[]}\left(\varepsilon,\mathcal{T}_j^\pm,L^2(P)\right)
}
\,d\varepsilon
<\infty.
\end{equation*}
Theorem~\ref{thm:bracketing-equicontinuity} therefore yields
asymptotic equicontinuity of the empirical process indexed by
$\mathcal{T}_j^\pm$, as stated in \eqref{eq:asymptotic-equicontinuity}.
\end{proof}

\subsection{\texorpdfstring{$L^2(P)$}{L2(P)}
continuity at the origin}
\label{subsec:L2-continuity-directional}

We next use the almost sure differentiability at $0$
established in
Lemma~\ref{lem:population-directional-identification}
to prove continuity of $T_x^\pm$ at $0$ in $L^2(P)$.

\begin{lemma}
\label{lem:L2-continuity-directional}
\begin{equation}
\label{eq:L2-continuity-directional}
\lim_{x\to 0}\|T_x^\pm-T_0^\pm\|_{P,2}=0.
\end{equation}
\end{lemma}

\begin{proof}
By Lemma~\ref{lem:population-directional-identification},
$\rho_q$ is differentiable at $0$ for $P$-almost every $q$.
Fix such a $q$, a coordinate $j$, and either sign.
For any fixed $0<s<2r/3$, the bounds
\eqref{eq:directional-secant-bounds} and continuity of
$\rho_q$ imply
\begin{align*}
\sigma\frac{\rho_q(-se_j)-\rho_q(0)}s-\frac{c(q)s}{2}
&\leq\liminf_{x\to0}\sigma T_{x,j}^\pm(q)\\
&\leq\limsup_{x\to0}\sigma T_{x,j}^\pm(q)\\
&\leq\sigma\frac{\rho_q(0)-\rho_q(se_j)}s+\frac{c(q)s}{2}.
\end{align*}
Letting $s\downarrow0$, the leftmost and rightmost sides tend
to $-\sigma\partial_j\rho_q(0)=\sigma T_{0,j}^\pm(q)$.
Therefore, it follows that
\begin{equation*}
T_x^\pm(q)\longrightarrow T_0^\pm(q)\qquad\text{as }x\to0
\end{equation*}
for $P$-almost every $q$.
Moreover, \eqref{eq:directional-field-bound} gives
\begin{equation*}
|T_x^\pm(q)-T_0^\pm(q)|^2\leq4mL(q)^2.
\end{equation*}
Since $L\in L^2(P)$, dominated convergence proves
\eqref{eq:L2-continuity-directional}.
\end{proof}

\subsection{Evaluation at the sample Fr\'echet mean}
\label{subsec:random-index-removal}

\begin{proposition}
\label{prop:random-index-removal}
We have
\begin{equation}
\label{eq:random-index-removal}
\mathbb{G}_nT_{x_n}^\pm-\mathbb{G}_nT_0^\pm
=o_{\mathbb{P}}(1)
\end{equation}
in $\mathbb{R}^m$.
\end{proposition}

\begin{proof}
Fix either sign.
By Lemma~\ref{lem:L2-continuity-directional}
and $x_n\xrightarrow{\mathbb{P}}0$,
\begin{equation*}
\|T_{x_n}^\pm-T_0^\pm\|_{P,2}
\xrightarrow{\mathbb{P}}0.
\end{equation*}
Here $\|T_{x_n}^\pm-T_0^\pm\|_{P,2}$ denotes the function
$x\mapsto\|T_x^\pm-T_0^\pm\|_{P,2}$ evaluated at $x_n$.
Joint measurability of $(x,q)\mapsto T_x^\pm(q)$
ensures that this function is Borel measurable
and that $\mathbb{G}_nT_{x_n}^\pm$ is a
$\mathbb{R}^m$-valued random variable.

Fix $j\in\{1,\ldots,m\}$.
Since
\begin{equation*}
\|T_{x_n,j}^\pm-T_{0,j}^\pm\|_{P,2}
\leq
\|T_{x_n}^\pm-T_0^\pm\|_{P,2},
\end{equation*}
we obtain
\begin{equation}
\label{eq:coordinate-random-index-L2}
\|T_{x_n,j}^\pm-T_{0,j}^\pm\|_{P,2}
\xrightarrow{\mathbb{P}}0.
\end{equation}

By Lemma~\ref{lem:polynomial-bracketing},
\eqref{eq:asymptotic-equicontinuity} holds with
$\mathcal{F}=\mathcal{T}_j^\pm$.
For every $\varepsilon>0$ and $\delta>0$,
\begin{equation*}
\mathbb{P}\left(\left|\mathbb{G}_n(T_{x_n,j}^\pm-T_{0,j}^\pm)\right|>\varepsilon\right)
\leq
\mathbb{P}\left(
\|T_{x_n,j}^\pm-T_{0,j}^\pm\|_{P,2}>\delta
\right)
+\mathbb{P}^{*}\left(
\sup_{\substack{
h_1,h_2\in\mathcal{T}_j^\pm \\
\|h_1-h_2\|_{P,2}\leq\delta
}}
\left|\mathbb{G}_n(h_1-h_2)\right|
>\varepsilon
\right).
\end{equation*}
For each fixed $\delta>0$, the first term tends to $0$
by \eqref{eq:coordinate-random-index-L2}.
By \eqref{eq:asymptotic-equicontinuity},
the limit superior of the second term tends to $0$
as $\delta\downarrow0$.
Therefore
\begin{equation*}
\mathbb{G}_n(T_{x_n,j}^\pm-T_{0,j}^\pm)
=o_{\mathbb{P}}(1).
\end{equation*}
Combining the coordinatewise
convergence gives
\begin{equation*}
\mathbb{G}_n(T_{x_n}^\pm-T_0^\pm)
=o_{\mathbb{P}}(1)
\qquad\text{in }\mathbb{R}^m.
\end{equation*}
This proves \eqref{eq:random-index-removal}
and hence \eqref{eq:random-index-removal-target} for both signs.
\end{proof}

\begin{proof}[Proof of Theorem~\ref{thm:main-gradient-clt}]
Lemma~\ref{lem:population-directional-identification}
shows that $\rho_q$ is differentiable at $0$
for $P$-almost every $q\in Q$.
Since $0$ minimizes $f$ and $L\in L^{2+\eta}(P)$, we also have
\begin{equation*}
\mathbb{E}\left[\nabla\rho_X(0)\right]=0,
\qquad
\mathbb{E}\left[|\nabla\rho_X(0)|^2\right]<\infty.
\end{equation*}
Moreover,
\begin{equation*}
T_0^+(q)=T_0^-(q)=-\nabla\rho_q(0)
\qquad\text{for }P\text{-almost every }q\in Q,
\qquad
PT_0^\pm=0.
\end{equation*}

Using $T_0^+=T_0^-$ $P$-almost surely,
\eqref{eq:empirical-process-gradient} gives, on $E_n$,
\begin{equation*}
\mathbb{G}_n(T_{x_n,j}^+-T_{0,j}^+)
\leq\sqrt{n}\,\partial_jf(x_n)-\mathbb{G}_nT_{0,j}^+
\leq\mathbb{G}_n(T_{x_n,j}^--T_{0,j}^-),
\qquad j=1,\ldots,m.
\end{equation*}
The leftmost and rightmost sides are $o_{\mathbb{P}}(1)$
by Proposition~\ref{prop:random-index-removal}.
Since $\mathbb{P}(E_n)\to1$ and $m<\infty$, it follows that
\begin{equation*}
\sqrt{n}\nabla f(x_n)
=
\mathbb{G}_nT_0^++o_{\mathbb{P}}(1).
\end{equation*}
Consequently,
\begin{equation*}
\sqrt{n}\nabla f(x_n)
=
-\frac1{\sqrt{n}}
\sum_{i=1}^n\nabla\rho_{X_i}(0)
+
o_{\mathbb{P}}(1),
\end{equation*}
which is \eqref{eq:main-asymptotic-expansion}.

The central limit theorem gives
\begin{equation*}
-\frac1{\sqrt{n}}
\sum_{i=1}^n\nabla\rho_{X_i}(0)
\xrightarrow{d}
\mathcal{N}_{\mathbb{R}^m}(0,\Sigma),
\qquad
\Sigma
=
\operatorname{Cov}\left(\nabla\rho_X(0)\right).
\end{equation*}
The conclusion \eqref{eq:main-gradient-clt}
now follows from Slutsky's theorem.
\end{proof}

\section{Proofs of the scaling theorems}
\label{sec:scaling-limits}

\begin{proof}[Proof of Theorem~\ref{thm:gradient-scaling-limit}]
For $y\in A_nB_{r}(0)$, set
\begin{equation*}
\Lambda_n(y)\coloneqq\sqrt{n}\nabla f\left(A_n^{-1}y\right).
\end{equation*}
By \eqref{eq:gradient-scaling-limit},
$\Lambda_n\to\Lambda$ locally uniformly on $\mathbb{R}^m$.

\medskip
\noindent\textbf{Step 1. Convergence of $A_{n+1}A_n^{-1}$.}
We prove that
\begin{equation}\label{eq:Cn_limit}
C_n\coloneqq A_{n+1}A_n^{-1}\longrightarrow I
\end{equation}
in operator norm, where $I$ is the identity map.

For $y\in A_nB_{r}(0)$,
\begin{equation}\label{eq:successive-scaling-identity}
\Lambda_{n+1}(C_ny)
=\sqrt{\frac{n+1}{n}}\,\Lambda_n(y).
\end{equation}
If $\{C_n\}$ were unbounded, we could pass to a subsequence
with $\|C_n\|_{\mathrm{op}}\to\infty$ and choose unit vectors
$v_n$ such that $|C_n^{-1}v_n|=\|C_n\|_{\mathrm{op}}^{-1}$.
Passing to a further subsequence, suppose that $v_n\to v$
with $|v|=1$.
Then $C_n^{-1}v_n\to0$, so local uniform convergence and
\eqref{eq:successive-scaling-identity}, applied with
$y=C_n^{-1}v_n$, give $\Lambda(v)=\Lambda(0)$.
This contradicts injectivity of $\Lambda$.
Thus $\{C_n\}$ is bounded.
For any subsequence along which $C_n\to C$, the same identity
and local uniform convergence give
$\Lambda(Cy)=\Lambda(y)$ for every $y\in\mathbb{R}^m$.
Injectivity implies $C=I$, proving \eqref{eq:Cn_limit}.

\medskip
\noindent\textbf{Step 2. A uniform bound near the origin.}
For each $M>0$, we show that there exist $\delta\in(0,r)$ and $R>0$
such that for every $n\geq1$ and every $u$ with $|u|\leq\delta$,
\begin{equation*}
\left|\sqrt n\,\nabla f(u)\right|\leq M
\quad\Longrightarrow\quad
|A_nu|<R.
\end{equation*}

Fix $M>0$. Since $\Lambda^{-1}$ is continuous, the set
$\Lambda^{-1}\left(\overline{B_{M+1}(0)}\right)$ is compact.
Choose $R>0$ such that $|\Lambda(w)|>M+1$ whenever $|w|\geq R$.
By local uniform convergence on the annulus
$\{w\mid R\leq|w|\leq 2R\}$ and by \eqref{eq:Cn_limit},
choose $N$ such that for every $k\geq N$,
\begin{equation*}
|\Lambda_k(w)|>M\quad\text{for }R\leq|w|\leq 2R,
\qquad
\|C_k\|_{\mathrm{op}}\leq 2.
\end{equation*}
Choose $\delta\in(0,r)$ so that $\|A_k\|_{\mathrm{op}}\delta<R$
for every $1\leq k\leq N$.

Suppose that $|u|\leq\delta$, $|\sqrt n\,\nabla f(u)|\leq M$,
and $|A_nu|\geq R$. Let $k$ be the first index with $|A_ku|\geq R$.
Then $N<k\leq n$ and $|A_{k-1}u|<R$, so
\begin{equation*}
R\leq|A_ku|\leq\|C_{k-1}\|_{\mathrm{op}}|A_{k-1}u|<2R.
\end{equation*}
Hence $|\Lambda_k(A_ku)|>M$. On the other hand,
\begin{equation*}
|\Lambda_k(A_ku)|
=\sqrt{\frac{k}{n}}\left|\sqrt n\,\nabla f(u)\right|
\leq M,
\end{equation*}
a contradiction.

\medskip
\noindent\textbf{Step 3. Tightness.}
We prove $\{A_n x_n\}_n$ is tight.

Theorem~\ref{thm:main-gradient-clt} gives
\begin{equation}\label{eq:gradient-limit-used-for-scaling}
\sqrt{n}\nabla f(x_n)\xrightarrow{d}\zeta.
\end{equation}
Given $\varepsilon>0$, choose $M>0$ by tightness so that
$\mathbb{P}(|\sqrt n\,\nabla f(x_n)|>M)<\varepsilon/2$ for every $n$.
Let $\delta,R$ be the constants from Step~2 for this $M$. Then
\begin{equation*}
\mathbb{P}(|A_nx_n|\geq R)
\leq \mathbb{P}(|x_n|>\delta)
+\mathbb{P}(|\sqrt n\,\nabla f(x_n)|>M).
\end{equation*}
Since $x_n\xrightarrow{\mathbb{P}}0$, this bound is less than
$\varepsilon$ for all sufficiently large $n$. This proves tightness
of $\{A_nx_n\}$.

\medskip
\noindent\textbf{Step 4. The limit.}
We prove \eqref{eq:scaled-frechet-mean-limit}.

By tightness from Step~3 and local uniform convergence,
\begin{equation*}
\Lambda_n(A_nx_n)-\Lambda(A_nx_n)
\xrightarrow{\mathbb{P}}0.
\end{equation*}
Since $x_n\in B_{r/3}(0)\subset B_r(0)$,
$\Lambda_n(A_nx_n)=\sqrt{n}\nabla f(x_n)$.
Together with \eqref{eq:gradient-limit-used-for-scaling},
Slutsky's theorem yields $\Lambda(A_nx_n)\xrightarrow{d}\zeta$.
The continuous mapping theorem, applied to $\Lambda^{-1}$,
now gives $A_nx_n\xrightarrow{d}\Lambda^{-1}(\zeta)$.
\end{proof}

\begin{proof}[Proof of Theorem~\ref{thm:inverse-gradient-scaling-limit}]
Since $\nabla f(0)=0$ and $\nabla f(V)$ is an open neighborhood of $0$,
every compact subset of $\mathbb{R}^m$ is contained in
$\sqrt n\nabla f(V)$ for all sufficiently large $n$.
For any $y\in D$ and any sequence $y_n\to y$, the points $y_n$
eventually lie in a compact subset of $D$.
Thus \eqref{eq:inverse-gradient-scaling-limit} and continuity of
$\Lambda$ give
\begin{equation*}
A_n(\nabla f|_V)^{-1}\left(\frac{y_n}{\sqrt n}\right)
\longrightarrow\Lambda(y).
\end{equation*}
Extend each map $y\mapsto A_n(\nabla f|_V)^{-1}(y/\sqrt n)$
by $0$ outside $\sqrt n\nabla f(V)$, and extend $\Lambda$ by $0$
outside $D$.
By Theorem~\ref{thm:main-gradient-clt},
$\sqrt n\nabla f(x_n)\xrightarrow{d}\zeta$.
Hence, by the extended continuous mapping theorem
\cite[Theorem~18.11]{vanDerVaart1998},
the images of these random vectors under the extended maps
converge in distribution to $\Lambda(\zeta)$.
On $\{x_n\in V\}$, these images equal $A_nx_n$.
Since $x_n\xrightarrow{\mathbb{P}}0$, this event has probability
tending to one, proving \eqref{eq:inverse-scaled-frechet-mean-limit}.
\end{proof}

\section{Riemannian manifolds}
\label{sec:riemannian-manifolds}

In this section, let $(M,g)$ be a complete connected
$m$-dimensional smooth Riemannian manifold without boundary,
and let $d$ be its geodesic distance.
We retain the notation of Section~\ref{sec:setting-main-result}.
Throughout this section, assume that
$\mathbb{E}\left[d(o,X)^2\right]<\infty$ for some $o\in M$
and that $F$ has a unique minimizer $\mu$.
We verify the assumptions of Theorem~\ref{thm:main-gradient-clt}
in this setting and give a sufficient condition for the local regularity
of $F$.

By the Hopf--Rinow theorem, $(M,d)$ is proper.
Consequently, $F$ and every $F_n$ have nonempty compact
minimizer sets.
Moreover, every measurable sequence of sample
Fr\'echet means satisfies
\begin{equation*}
d(\mu_n,\mu)\longrightarrow0
\qquad\text{almost surely};
\end{equation*}
see \cite{BhattacharyaPatrangenaru2003}.
Thus \textup{(\ref{ass:local-coordinates-consistency})}
holds.

Fix an identification $T_\mu M\simeq\mathbb{R}^m$,
choose $0<r<\operatorname{inj}(\mu)$, and take
\begin{equation*}
\phi=\exp_\mu^{-1}\colon B_{r}^M(\mu)
\longrightarrow B_{r}(0).
\end{equation*}
Here $B_s^M(p)$ denotes the geodesic ball of radius $s$
centered at $p$, whereas $B_s(x)$ denotes a ball in
$T_\mu M\simeq\mathbb{R}^m$.
This verifies \textup{(\ref{ass:locally-Euclidean})}.
In these coordinates,
\begin{equation*}
\rho_q(x)=\frac12d(\exp_\mu x,q)^2,
\qquad
f(x)=F(\exp_\mu x),
\qquad x\in B_{r}(0).
\end{equation*}
The gradient $\nabla f$ is taken with respect to the
fixed inner product $g_\mu$ on $T_\mu M$.
For the standard results from Riemannian geometry used below, see \cite{Lee2018}.

\subsection{Semiconcavity and local Lipschitz continuity}
\label{subsec:uniform-semiconcavity}

Set
\begin{equation*}
b(q)\coloneqq1+d(\mu,q),\qquad q\in M.
\end{equation*}
By geodesic semiconcavity of squared distance functions
(see \cite[Lemma~3.1]{LytchakHuckemann2025}) and the compactness of
$\overline{B_{r}^M(\mu)}$, there exists a constant $C>0$ independent of
$q$ such that
\begin{equation*}
t\longmapsto\frac12d(\gamma(t),q)^2-\frac{Cb(q)}2t^2
\end{equation*}
is concave for every unit-speed geodesic $\gamma$
contained in a fixed open neighborhood of this compact ball.
The following lemma obtains the corresponding estimates
in normal coordinates by bounding the first and second
derivatives of the exponential map.
These estimates verify
\textup{(\ref{ass:semiconcavity-moments})}
under a $(2+\eta)$-moment condition.

\begin{lemma}
\label{lem:uniform-semiconcavity}
There exists a constant $C_0>0$ such that, for every
$q\in M$, the function
\begin{equation*}
x\longmapsto\rho_q(x)-\frac{C_0b(q)}2|x|^2
\end{equation*}
is concave on $B_{r}(0)$, and
\begin{equation}
\left|\rho_q(x)-\rho_q(y)\right|
\leq C_0b(q)|x-y|,
\qquad x,y\in B_{r}(0).
\label{eq:uniform-loss-lipschitz}
\end{equation}
In particular, if $\rho_q$ is differentiable at $x\in B_r(0)$, then
\begin{equation}
\rho_q(y)-\rho_q(x)-\langle\nabla\rho_q(x),y-x\rangle
\leq \frac{C_0b(q)}{2}|y-x|^2,
\qquad y\in B_r(0).
\label{eq:uniform-loss-semiconcavity}
\end{equation}
\end{lemma}

\begin{proof}
Write $E=\exp_\mu$.
Constants denoted by $C$ below are independent of $x$
and $q$ and may increase from line to line.
For $z,w\in B_{r+1}^M(\mu)$, the triangle inequality gives
\begin{equation*}
\left|\frac12d(z,q)^2-\frac12d(w,q)^2\right|
\leq Cb(q)d(z,w).
\end{equation*}
Since $E$ is Lipschitz on $B_{r}(0)$, this also gives
\eqref{eq:uniform-loss-lipschitz} for $C_0$ sufficiently large.

Fix $x\in B_{r}(0)$ and a unit vector $v\in T_\mu M$, and set
\begin{equation*}
\gamma(t)=\exp_{E(x)}(t\,dE_xv).
\end{equation*}
Since $\exp_{E(x)}^{-1}(\gamma(\pm t))=\pm t\,dE_xv$ and
$\exp_{E(x)}^{-1}(E(x\pm tv))=\pm t\,dE_xv+O(t^2)$, we obtain
\begin{equation*}
d(E(x\pm tv),\gamma(\pm t))
\leq C\left|\exp_{E(x)}^{-1}(E(x\pm tv))-\exp_{E(x)}^{-1}(\gamma(\pm t))\right|
\leq Ct^2.
\end{equation*}
Also, $|\dot\gamma|=|dE_xv|\leq C$ by compactness.
Using the preceding Lipschitz bound and the geodesic
concavity stated before the lemma, we obtain
\begin{align*}
\rho_q(x+tv)+\rho_q(x-tv)-2\rho_q(x)
&\leq
\frac12d(\gamma(t),q)^2+\frac12d(\gamma(-t),q)^2
-d(E(x),q)^2+Cb(q)t^2\\
&\leq Cb(q)t^2.
\end{align*}
Thus, after increasing $C_0$ if necessary,
$\rho_q-C_0b(q)|\cdot|^2/2$ is locally midpoint concave
on every Euclidean line segment in $B_{r}(0)$.
By continuity, its restriction to each such segment
is concave.
For $x,y\in B_r(0)$ and $0<t<1$, concavity gives
\begin{align*}
(1-t)\rho_q(x)+t\rho_q(y)-\rho_q(x+t(y-x))
&\leq \frac{C_0b(q)}{2}
\left((1-t)|x|^2+t|y|^2-|x+t(y-x)|^2\right)\\
&=\frac{C_0b(q)}{2}t(1-t)|y-x|^2.
\end{align*}
Dividing by $t$ gives
\begin{equation*}
\rho_q(y)-\rho_q(x)
-\frac{\rho_q(x+t(y-x))-\rho_q(x)}{t}
\leq \frac{C_0b(q)}{2}(1-t)|y-x|^2.
\end{equation*}
If $\rho_q$ is differentiable at $x$, letting $t\downarrow0$
yields \eqref{eq:uniform-loss-semiconcavity}.
\end{proof}

In particular, suppose that, for some $\eta>0$,
\begin{equation}
\mathbb{E}\left[d(o,X)^{2+\eta}\right]<\infty.
\label{eq:higher-moment-assumption}
\end{equation}
Then $b\in L^{2+\eta}(P)$, and
\textup{(\ref{ass:semiconcavity-moments})} holds with
\begin{equation*}
c(q)=C_0b(q),\qquad L(q)=C_0b(q).
\end{equation*}
Thus, under uniqueness of the population mean and
\eqref{eq:higher-moment-assumption}, assumptions
\textup{(\ref{ass:local-coordinates-consistency})}--%
\textup{(\ref{ass:semiconcavity-moments})} are satisfied.

\subsection{The central limit theorem}
\label{subsec:riemannian-gradient-clt}

Whenever there is a unique minimizing geodesic
$\gamma\colon[0,1]\to M$ from $\mu$ to $q$, parametrized
with constant speed, write $\log_\mu(q)=\dot\gamma(0)$.
This agrees with the usual logarithm map on
$M\setminus\operatorname{Cut}(\mu)$.
The next theorem recovers almost sure uniqueness of the
minimizing geodesic.
The stronger conclusion $P(\operatorname{Cut}(\mu))=0$
already follows from
\cite[Theorem~1.4]{LytchakHuckemann2025} under a finite
second moment.

\begin{theorem}
\label{thm:riemannian-gradient-clt}
Assume that $\mathbb{E}\left[ d(o, X)^{2+\eta}\right]<\infty$ for some $\eta>0$
and suppose that $F\in C^{1,\alpha}(U)$ for some $\alpha\in(0,1]$
and some open neighborhood $U$ of $\mu$.
Fix $0<r<\operatorname{inj}(\mu)$ such that
$\overline{B_r^M(\mu)}\subset U$.
For any sequence $\{\mu_n\}_n$ of measurable sample Fr\'echet means, set
\begin{equation*}
x_n\coloneqq
\begin{cases}
\log_\mu(\mu_n),&\mu_n\in B_{r/3}^M(\mu),\\
0,&\mu_n\notin B_{r/3}^M(\mu).
\end{cases}
\end{equation*}
Then the minimizing geodesic from $\mu$ to $X$ is
almost surely unique, and
\begin{equation*}
\mathbb{E}\left[\log_\mu(X)\right]=0,
\qquad
\mathbb{E}\left[|\log_\mu(X)|^2\right]<\infty.
\end{equation*}
Moreover,
\begin{equation}
\sqrt{n}\nabla f(x_n)
=\frac1{\sqrt{n}}\sum_{i=1}^n\log_\mu(X_i)
+o_{\mathbb{P}}(1)
\label{eq:riemannian-asymptotic-expansion}
\end{equation}
in $T_\mu M$.
Consequently,
\begin{equation*}
\sqrt{n}\nabla f(x_n)
\xrightarrow{d}\mathcal{N}_{T_\mu M}(0,\Sigma),
\qquad
\Sigma=\operatorname{Cov}\left(\log_\mu(X)\right).
\end{equation*}
\end{theorem}

\begin{proof}
The preceding discussion verifies
\textup{(\ref{ass:local-coordinates-consistency})}--%
\textup{(\ref{ass:semiconcavity-moments})}.
Since $\overline{B_r^M(\mu)}\subset U$, the smoothness
of $\exp_\mu$ and the assumed regularity of $F$ give
$f\in C^{1,\alpha}(B_r(0))$.
Thus \textup{(\ref{ass:population-regularity})} holds
on $B_r(0)$, and Theorem~\ref{thm:main-gradient-clt}
applies to the localized sample means defined above.

Theorem~\ref{thm:main-gradient-clt} gives differentiability
of $\rho_q$ at $0$ for $P$-almost every $q$.
Equivalently, $p\mapsto d(p,q)^2/2$ is differentiable at
$\mu$, which implies that the minimizing geodesic from
$\mu$ to $q$ is unique; see
\cite[Subsection~1.9]{FigalliVillani2011}.
The first variation formula and
$d(\exp_\mu)_0=\operatorname{id}_{T_\mu M}$ then give
\begin{equation*}
\nabla\rho_q(0)=-\log_\mu(q)
\qquad\text{for }P\text{-almost every }q.
\end{equation*}
The expectation and moment identities,
\eqref{eq:riemannian-asymptotic-expansion}, and the
stated Gaussian limit now follow from
Theorem~\ref{thm:main-gradient-clt}.
\end{proof}

\subsection{Regularity from probability bounds near cut loci}
\label{subsec:frechet-c1alpha}
The regularity of the Fr\'echet function $F$ depends on both
the distribution $P$ and the geometry of $M$, and the
$C^{1,\alpha}$ regularity assumed in
Theorem~\ref{thm:riemannian-gradient-clt}
may be difficult to verify directly.
We therefore give a sufficient condition for \textup{(\ref{ass:population-regularity})} expressed more
explicitly in terms of the distribution and the geometry,
through probability bounds for neighborhoods of cut loci.
For $q\in M$, let $\operatorname{Cut}(q)$ denote its
cut locus, with $\dist(p,\varnothing)=+\infty$.

\begin{proposition}
\label{prop:frechet-c1alpha-cut-neighborhood}
Assume $\mathbb{E}\left[d(o,X)^{2+\eta}\right]<\infty$ for some $\eta>0$.
Suppose that there are an open neighborhood $U$ of $\mu$,
constants $C,t_0>0$, and an exponent $\beta\in(0,1]$ such that
\begin{equation}
\sup_{p\in U}
\mathbb{P}\left(\dist\left(p,\operatorname{Cut}(X)\right)
\leq t\right)\leq Ct^\beta,
\qquad 0<t<t_0.
\label{eq:c1alpha-cut-probability}
\end{equation}
Then, with $\alpha_0\coloneqq\beta(1+\eta)/(2+\eta)$, there exists
$r'\in(0,r)$ such that
\begin{equation*}
f\in C^{1,\alpha_0}(B_{r'}(0)),
\qquad
F\in C^{1,\alpha_0}(B_{r'}^M(\mu)).
\end{equation*}
In particular, the conclusions of
Theorem~\ref{thm:riemannian-gradient-clt} hold.
\end{proposition}

\begin{proof}
Choose $0<r'<\min\{r/3,1/2\}$ such that
$\overline{B_{r'}^M(\mu)}\subset U$, and write $E=\exp_\mu$.
For $x\in B_{r'}(0)$ and $q\in M$, set
\begin{equation*}
\delta_q(x)\coloneqq
1\wedge\dist\left(E(x),\operatorname{Cut}(q)\right).
\end{equation*}
Fix $x\in B_{r'}(0)$.
For each $t\geq0$, the set
\begin{equation*}
\{q\in M\mid \dist(E(x),\operatorname{Cut}(q))\leq t\}
\end{equation*}
is closed. Indeed, if $q_j\to q$ with $q_j$ in this set,
choose $z_j\in\operatorname{Cut}(q_j)$ with $d(E(x),z_j)\leq t+1/j$.
By compactness of $\overline{B_{t+1}^M(E(x))}$, a subsequence
of $z_j$ converges to a point $z$ with $d(E(x),z)\leq t$.
Since the cut relation is closed, $z\in\operatorname{Cut}(q)$.
Thus $\dist(E(x),\operatorname{Cut}(q))\leq t$.
It follows that $q\mapsto\dist(E(x),\operatorname{Cut}(q))$,
and hence $q\mapsto\delta_q(x)$, is Borel measurable.
Constants denoted by $C$ below are independent of
$x,y,q$ and may increase from line to line.

\medskip
\noindent\textbf{Step 1. Gradient estimates for $\rho_q$.}
For $x,y\in B_{r'}(0)$ with
$\delta_q(x),\delta_q(y)>0$, we show that
\begin{equation}
|\nabla\rho_q(x)|\leq Cb(q),
\qquad
|\nabla\rho_q(y)-\nabla\rho_q(x)|
\leq Cb(q)\min\left\{1,\frac{|x-y|}{\delta_q(x)}\right\}.
\label{eq:c1alpha-loss-gradient-bounds}
\end{equation}

The first bound follows from
Lemma~\ref{lem:uniform-semiconcavity}.
For the second, fix $x$ and $q$ and define, for $u\in B_r(0)$,
\begin{equation*}
G(u)\coloneqq
\rho_q(u)-\rho_q(x)-\langle\nabla\rho_q(x),u-x\rangle.
\end{equation*}
Then $G(x)=0$ and $\nabla G(x)=0$, and
$\nabla G(y)=\nabla\rho_q(y)-\nabla\rho_q(x)$.
We first show that
\begin{equation}
G(u)\geq-\frac{Cb(q)}{\delta_q(x)}|u-x|^2,
\qquad u\in B_r(0).
\label{eq:c1alpha-loss-lower-support}
\end{equation}
Set $\ell=d(q,E(x))$. If $\ell=0$, then $\rho_q$ attains
its minimum $0$ at $x$, so $G=\rho_q\geq0$.
For $\ell>0$, take
\begin{equation*}
a\coloneqq\frac12\min\left\{\delta_q(x),
\min_{w\in\overline{B_{r'+1}^M(\mu)}}\operatorname{inj}(w)\right\}.
\end{equation*}
Extend the minimizing geodesic from $q$ to $E(x)$ by
length $a$ to a point $z$.
Since $a<\dist(E(x),\operatorname{Cut}(q))$, this extension
is minimizing, so $d(q,z)=\ell+a$ and $d(E(x),z)=a$.
By the triangle inequality,
\begin{equation*}
b(z)\leq 1+d(\mu,E(x))+a<2+r'.
\end{equation*}
Since $a<1$, we have $z\in B_{r'+1}^M(\mu)$ and hence
$a<\operatorname{inj}(z)$ by the choice of $a$.
Cauchy--Schwarz inequality and the triangle
inequality give, for every $u\in B_r(0)$,
\begin{equation*}
H(u)\coloneqq\rho_q(u)+\frac{\ell}{a}\rho_z(u)\geq\frac{\ell}{2(\ell+a)}
\left(d(q,E(u))+d(z,E(u))\right)^2 \geq\frac{\ell(\ell+a)}2,
\end{equation*}
with equality at $u=x$.
The function $H$ is smooth near $x$.
Indeed, $\delta_q(x)>0$ gives $E(x)\notin\operatorname{Cut}(q)$,
and $d(E(x),z)=a<\operatorname{inj}(z)$ gives
$E(x)\notin\operatorname{Cut}(z)$.
Since $H$ attains its minimum at $x$, we have $\nabla H(x)=0$.
Thus, for every $u\in B_r(0)$,
\begin{align*}
0 &\leq H(u)-H(x)-\langle\nabla H(x),u-x\rangle\\
  &=G(u)+\frac{\ell}{a}
\left(\rho_z(u)-\rho_z(x)-\langle\nabla\rho_z(x),u-x\rangle\right).
\end{align*}
Rearranging gives
\begin{equation*}
G(u)\geq
-\frac{\ell}{a}
\left(\rho_z(u)-\rho_z(x)-\langle\nabla\rho_z(x),u-x\rangle\right).
\end{equation*}
Applying \eqref{eq:uniform-loss-semiconcavity} to $\rho_z$ at $x$
and using $b(z)\leq 2+r'$ and $\ell\leq Cb(q)$ now proves
\eqref{eq:c1alpha-loss-lower-support}.

Suppose $\nabla G(y)\ne0$. Then $x\ne y$ since $\nabla G(x)=0$.
Define
\begin{equation*}
u=y-|x-y|\frac{\nabla G(y)}{|\nabla G(y)|}.
\end{equation*}
Then $|u|\leq |y|+|x-y|<3r'<r$ and $|u-x|\leq2|x-y|$.
Subtracting an affine function leaves the remainder in
\eqref{eq:uniform-loss-semiconcavity} unchanged, so the inequality also
holds for $G$. At $x$, it gives $G(y)\leq Cb(q)|x-y|^2$
because $G(x)=0$ and $\nabla G(x)=0$. At $y$, it yields
\begin{align*}
G(u)&\leq G(y)+\langle\nabla G(y),u-y\rangle+Cb(q)|u-y|^2\\
&\leq -|x-y|\,|\nabla G(y)|+Cb(q)|x-y|^2.
\end{align*}
The lower bound \eqref{eq:c1alpha-loss-lower-support} and
$|u-x|\leq2|x-y|$ give
\begin{equation*}
G(u)\geq -\frac{Cb(q)}{\delta_q(x)}|x-y|^2.
\end{equation*}
Comparing the two bounds gives
\begin{equation*}
|x-y|\,|\nabla G(y)|\leq
Cb(q)\left(1+\frac{1}{\delta_q(x)}\right)|x-y|^2.
\end{equation*}
Dividing by $|x-y|$ and using $\delta_q(x)\leq1$ yields
$|\nabla G(y)|\leq Cb(q)|x-y|/\delta_q(x)$.
The same inequality holds when $\nabla G(y)=0$.
The bound $|\nabla\rho_q|\leq Cb(q)$ at $x$ and $y$ also gives
$|\nabla G(y)|\leq 2Cb(q)$ by the triangle inequality.
Taking the smaller of the two bounds proves the second bound in
\eqref{eq:c1alpha-loss-gradient-bounds}.

\medskip
\noindent\textbf{Step 2. An estimate near the cut locus.}
We show that there is a constant $C>0$ such that
\begin{equation}
\sup_{x\in B_{r'}(0)}
\mathbb{E}\left[b(X)\mathbf{1}_{\{\delta_X(x)\leq t\}}\right]
\leq Ct^{\alpha_0},
\qquad 0<t\leq1.
\label{eq:c1alpha-weighted-cut-probability}
\end{equation}

For $x\in B_{r'}(0)$ and $0<t<1\wedge t_0$, H\"older's inequality
with conjugate exponents $2+\eta$ and $(2+\eta)/(1+\eta)$ gives
\begin{equation*}
\mathbb{E}\left[b(X)\mathbf{1}_{\{\delta_X(x)\leq t\}}\right]
\leq \|b\|_{L^{2+\eta}(P)}
\mathbb{P}(\delta_X(x)\leq t)^{(1+\eta)/(2+\eta)}
\leq Ct^{\alpha_0}.
\end{equation*}
The last inequality uses $b\in L^{2+\eta}(P)$ and
\eqref{eq:c1alpha-cut-probability} applied at $E(x)\in U$.
For $1\wedge t_0\leq t\leq1$, use $\mathbb{E}[b(X)]<\infty$ and enlarge $C$.
In particular, $\delta_X(x)>0$ almost surely for each
fixed $x$.

\medskip
\noindent\textbf{Step 3. Differentiability of $f$ and H\"older continuity of $\nabla f$.}
We show that $f\in C^{1,\alpha_0}(B_{r'}(0))$.

For each fixed $x\in B_{r'}(0)$, we have $\delta_X(x)>0$ almost surely,
so $\rho_X$ is differentiable at $x$ almost surely. The Lipschitz bound
\eqref{eq:uniform-loss-lipschitz} gives, for $y\ne x$ in $B_{r'}(0)$,
\begin{equation*}
\frac{\left|\rho_X(y)-\rho_X(x)-\langle\nabla\rho_X(x),y-x\rangle\right|}
{|y-x|}
\leq 2C_0b(X).
\end{equation*}
The left-hand side tends to zero almost surely as $y\to x$. Since
$b(X)$ is integrable, taking expectations and applying dominated
convergence shows that $f$ is differentiable at $x$ with
\begin{equation*}
\nabla f(x)=\mathbb{E}[\nabla\rho_X(x)].
\end{equation*}

Now fix distinct $x,y\in B_{r'}(0)$ and set $s=|x-y|<1$.
Splitting according to whether $\delta_X(x)\leq s$,
\eqref{eq:c1alpha-loss-gradient-bounds} and
\eqref{eq:c1alpha-weighted-cut-probability} give
\begin{equation*}
\mathbb{E}\left[
|\nabla\rho_X(y)-\nabla\rho_X(x)|
\mathbf{1}_{\{\delta_X(x)\leq s\}}\right]
\leq C\mathbb{E}\left[b(X)
\mathbf{1}_{\{\delta_X(x)\leq s\}}\right]
\leq Cs^{\alpha_0}.
\end{equation*}
On $\{\delta_X(x)>s\}$, we have
\begin{equation*}
\frac{1}{\delta_X(x)}=1+\int_{\delta_X(x)}^1\frac{dt}{t^2}
=1+\int_s^1\frac{\mathbf{1}_{\{\delta_X(x)\leq t\}}}{t^2}\,dt.
\end{equation*}
Multiplying by $b(X)\mathbf{1}_{\{\delta_X(x)>s\}}$ and applying
Tonelli's theorem to the nonnegative integrand gives
\begin{align*}
\mathbb{E}\left[
\frac{b(X)}{\delta_X(x)}
\mathbf{1}_{\{\delta_X(x)>s\}}\right]
&=\mathbb{E}\left[b(X)\mathbf{1}_{\{\delta_X(x)>s\}}\right]
+\int_s^1\frac{\mathbb{E}\left[b(X)\mathbf{1}_{\{s<\delta_X(x)\leq t\}}\right]}
{t^2}\,dt\\
&\leq \mathbb{E}[b(X)]+C\int_s^1t^{\alpha_0-2}\,dt\\
&= \mathbb{E}[b(X)]+C\frac{s^{\alpha_0-1}-1}{1-\alpha_0}
\leq Cs^{\alpha_0-1},
\end{align*}
where we used \eqref{eq:c1alpha-weighted-cut-probability} and
$0<\alpha_0<1$. By \eqref{eq:c1alpha-loss-gradient-bounds}, the
contribution on $\{\delta_X(x)>s\}$ is at most $Cs$ times the
expectation just bounded, hence at most $Cs^{\alpha_0}$.
Adding the two estimates gives
\begin{equation*}
|\nabla f(y)-\nabla f(x)|
\leq \mathbb{E}\left[|\nabla\rho_X(y)-\nabla\rho_X(x)|\right]
\leq C|x-y|^{\alpha_0}.
\end{equation*}
Thus $f\in C^{1,\alpha_0}(B_{r'}(0))$, and the smooth
normal coordinate chart gives
$F\in C^{1,\alpha_0}(B_{r'}^M(\mu))$.
\end{proof}

\section{Examples}
\label{sec:examples}
In this section, we apply our theorems to study the limiting behavior of
sample Fr\'echet means in several settings. We begin with a hexagonal flat
torus example in which Theorem~\ref{thm:gradient-scaling-limit} gives
an anisotropic limit.

We then consider a variant of the logarithmic smeariness example of
\cite[Example~2]{HundrieserEltznerHuckemann2026}.
The assumptions of Theorem~\ref{thm:gradient-scaling-limit} fail in
this example, while Theorem~\ref{thm:inverse-gradient-scaling-limit}
applies and recovers the logarithmic limit established there.

In the final example, Theorem~\ref{thm:riemannian-gradient-clt}
gives a Gaussian limit for $\sqrt{n}\nabla f(x_n)$. No deterministic
linear normalization of $x_n$ yields a nondegenerate limit because
the gradient of the Fr\'echet function oscillates near the origin
and these oscillations persist under rescaling.

\subsection{Anisotropic scaling on a flat torus}
\begin{example}
\label{ex:hexagonal-anisotropic-torus}
There is a probability distribution on a flat torus with
the origin as its unique population Fr\'echet mean. If $(x_n,y_n)$
are the local Euclidean coordinates of a measurable sample
Fr\'echet mean, then
\begin{equation*}
  \left(\sqrt{n}\,x_n,n^{1/4}y_n\right)
  \xrightarrow{d}
  \left(\frac32\zeta_1-\frac{\sqrt{3}}4|\zeta_2|,
         \operatorname{sgn}(\zeta_2)\sqrt{|\zeta_2|}\right),
\end{equation*}
where $\zeta_1$ and $\zeta_2$ are independent centered Gaussian
variables with variances $17/384$ and $17/128$, respectively.
The first component has the classical $\sqrt{n}$ scaling and a
non-Gaussian limit. The components of the limit are dependent.
\end{example}
\begin{proof}
Let $H\subset\mathbb{R}^2$ be the closed regular hexagon of side length $1$.
Choose coordinates so that $H$ is centered at the origin and $(1,0)$ is
normal to a pair of opposite sides. Identify each pair of opposite sides
by translation. The resulting quotient $M$ is a flat torus diffeomorphic
to $S^1\times S^1$. Write $d$ for its Riemannian distance.

Set $\mathcal{V}=\{\pm e_0,\pm e_+,\pm e_-\}$, where
$e_0=(1,0)$ and $e_\pm=(1/2,\pm\sqrt{3}/2)$. For $z,w\in M$, use the
same symbols for representatives in $H$. 
The distance $d(z,w)$ is the minimum of the Euclidean distances from
$z-w$ to the points in $\{0\}\cup\sqrt{3}\mathcal{V}$. Expanding the
squares gives
\begin{align}
  d(z,w)^2
  &= \min_{v\in\{0\}\cup\sqrt{3}\mathcal{V}} |z-w-v|^2\notag \\
  &= |z-w|^2 - 2\sqrt{3}
    \left(\max_{v\in\mathcal{V}}
      \langle z-w,v\rangle-\frac{\sqrt{3}}{2}\right)_+.\label{eq:distance-maximum}
\end{align}

Let $c\in[0,1)$ and let $Z$ take values in
$\{0,\pm e_+,\pm e_-\}$ with
\begin{equation*}
  \mathbb{P}(Z=0)=c,\qquad
  \mathbb{P}(Z=v)=\frac{1-c}{4},
  \quad v\in\{\pm e_+,\pm e_-\}.
\end{equation*}
Let $(U_1,U_2)$ be independent of $Z$, with both $U_i$ taking values in
$[\sqrt{3}/4,\sqrt{3}/2]$. Define
\begin{equation*}
  Y=U_1\mathbf{1}_{\{Z\in\{e_+,e_-\}\}}
    +U_2\mathbf{1}_{\{Z\in\{-e_+,-e_-\}\}},\qquad X=YZ.
\end{equation*}
We view $X$ as an $M$-valued random variable via its representative in $H$.

Fix $z\in B_{1/2}(0)$. On $\{Z=0\}$, we have $X=0$ and
$d(z,X)^2=|z|^2$. On $\{Z\ne0\}$, any $v\in\mathcal V$ with
$\langle Z,v\rangle\geq0$ satisfies
$\langle z-X,v\rangle\leq |z|<\sqrt{3}/2$.
The remaining directions other than $-Z$ satisfy
$\langle Z,v\rangle=-1/2$. If such a direction attained the maximum
in \eqref{eq:distance-maximum} with value greater than $\sqrt{3}/2$, then
comparison with $-Z$ would give
\begin{equation*}
  \langle z,v\rangle+\frac{Y}{2}>\frac{\sqrt{3}}{2},
  \qquad
  \langle z,v+Z\rangle-\frac{Y}{2}\geq0.
\end{equation*}
Adding these inequalities gives
$\langle z,2v+Z\rangle>\sqrt{3}/2$, which is impossible because
$|2v+Z|=\sqrt{3}$ and $|z|<1/2$. Thus whenever the correction term
is positive, the maximum is attained at $v=-Z$.
Since $Y=0$ on $\{Z=0\}$, both cases give
\begin{equation*}
  d(z,X)^2
  = |z-X|^2 - 2\sqrt{3}
    \left(Y-\langle z,Z\rangle-\frac{\sqrt{3}}{2}\right)_+.
\end{equation*}

For $z=(x,y)$, write $F(x,y)=\tfrac12\mathbb{E}[d((x,y),X)^2]$ and
$t_\pm=\langle z,e_\pm\rangle=(x\pm\sqrt{3}\,y)/2$. Define
\begin{equation*}
  g_i(u)=\mathbb{E}\left[\left(U_i-\frac{\sqrt{3}}{2}-u\right)_+\right],
  \qquad i=1,2.
\end{equation*}
The correction term vanishes at the origin. For $|z|<1/2$,
conditioning on $Z$ gives
\begin{equation*}
  F(x,y)-F(0,0)
  =\frac{x^2+y^2}{2}
   -\frac{1-c}{4}\left(\mathbb{E}[U_1]-\mathbb{E}[U_2]\right)x
   -\frac{\sqrt{3}(1-c)}{4}
      \sum_{\sigma\in\{+,-\}}\left(g_1(t_\sigma)+g_2(-t_\sigma)\right).
\end{equation*}
Since $g_i(u)=0$ for $u\geq0$, each summand is $g_1(-|t_\sigma|)$
when $t_\sigma<0$ and $g_2(-|t_\sigma|)$ when $t_\sigma>0$.

Choose the laws of $U_1$ and $U_2$ by setting, for $0\leq s\leq1$,
\begin{equation*}
  \mathbb{P}\left(U_i>\frac{\sqrt{3}}{2}-\frac{\sqrt{3}}{4}s\right)
  =2s-s^2+\frac{(-1)^{i-1}}{2}s^2(1-s)^2(1-2s),
  \qquad i=1,2.
\end{equation*}
The right-hand side increases from $0$ to $1$. Its derivative is
$2(1-s)(1+(-1)^{i-1}s(1-5s+5s^2)/2)\geq0$
on $[0,1]$. The third term integrates to zero because
$s^2(1-s)^2(1-2s)$ is antisymmetric about $1/2$. Integrating the two
tails therefore gives $\mathbb{E}[U_1]=\mathbb{E}[U_2]=5\sqrt{3}/12$,
so the linear term vanishes.

For $0\leq t\leq\sqrt{3}/4$, integration gives
\begin{align*}
  g_i(-t)
  &=\int_0^t\mathbb{P}\left(U_i>\frac{\sqrt{3}}{2}-s\right)\,ds\\
  &=\frac{4}{\sqrt{3}}t^2-\frac{16}{9}t^3
     +\frac{(-1)^{i-1}8}{9}t^3
        \left(1-\frac{4t}{\sqrt{3}}\right)^3.
\end{align*}
Using $t_+^2+t_-^2=(x^2+3y^2)/2$, we obtain the quadratic part
$\frac12(cx^2+(3c-2)y^2)$. Taking $c=2/3$ cancels the $y^2$ term and
gives, for $x^2+y^2<3/16$,
\begin{equation}\label{eq:expansion-Frechet}
  F(x,y)-F(0,0)
  =\frac{x^2}{3}+\frac{4\sqrt{3}}{27}
    \sum_{\sigma\in\{+,-\}}
    \left[
      |t_\sigma|^3+\frac12 t_\sigma^3
          \left(1-\frac{4|t_\sigma|}{\sqrt{3}}\right)^3
    \right].
\end{equation}
In particular, $\nabla F(0,0)=0$. The two cubic terms involving
$t_\sigma^3$ cancel along the $y$-axis, where
$F(0,y)-F(0,0)=|y|^3/3$ for $|y|<\sqrt{3}/4$.

To show that the origin is the unique global minimizer, observe that
$s^2(1-s)^2|1-2s|\leq s^2$ on $[0,1]$. Thus
\begin{equation*}
  \mathbb{P}\left(U_i>\frac{\sqrt{3}}{2}-s\right)
  <\frac{8s}{\sqrt{3}}
  \qquad (s>0,\ i=1,2).
\end{equation*}
For $0<s\leq\sqrt{3}/4$ this follows from the tail formula, and for
$s>\sqrt{3}/4$ the right-hand side exceeds $1$.
Integration gives $g_i(-t)<4t^2/\sqrt{3}$ for all $t>0$.
If $0<|z|<1/2$, at least one of $t_+$ and $t_-$ is nonzero. With
$c=2/3$, it follows that
\begin{equation*}
  F(x,y)-F(0,0)
  >
  \frac{x^2+y^2}{2}-\frac{t_+^2+t_-^2}{3}
  =\frac{x^2}{3}\geq0.
\end{equation*}
For $z\in H$ with $|z|\geq1/2$, use $d(z,0)=|z|$ and the reverse
triangle inequality to obtain
\begin{equation*}
  F(x,y)-F(0,0)
  \geq \frac12\mathbb{E}\left[(|z|-d(0,X))^2-d(0,X)^2\right]
  = \frac{|z|^2}{2}-|z|\,\mathbb{E}[d(0,X)].
\end{equation*}
Since $d(0,X)=Y$ and both $U_i$ have mean $5\sqrt{3}/12$, we have
$\mathbb{E}[d(0,X)]=5\sqrt{3}/36$. Hence the right-hand side is
$\frac{|z|}{2}(|z|-5\sqrt{3}/18)>0$, since $5\sqrt{3}/18<1/2$.
This proves that the origin is the unique population Fr\'echet mean.

\eqref{eq:expansion-Frechet} gives $F\in C^2$ on $B_{\sqrt{3}/4}(0)$.
Differentiating gives, as $z\to0$,
\begin{equation*}
  \nabla F(x,y)
  =\begin{pmatrix}2x/3\\0\end{pmatrix}
  +\frac{4\sqrt{3}}{9}
     \sum_{\sigma\in\{+,-\}}
     \left(t_\sigma|t_\sigma|+\frac12t_\sigma^2\right)e_\sigma
  +O(|z|^3).
\end{equation*}
Let $z_n=(x_n,y_n)$ denote the measurable sample Fr\'echet means.
By Theorem~\ref{thm:riemannian-gradient-clt},
\begin{equation*}
  \sqrt{n}\nabla F(z_n)\xrightarrow{d}\zeta,\qquad
  \zeta\sim\mathcal{N}_{\mathbb{R}^2}(0,\Sigma).
\end{equation*}
Since $(\mathbb{E}[U_1^2]+\mathbb{E}[U_2^2])/2=17/32$ and
$\mathbb{E}[X]=0$, conditioning on $Z$ yields
\begin{equation*}
  \Sigma=\operatorname{Cov}(X)
  =(1-c)\frac{\mathbb{E}[U_1^2]+\mathbb{E}[U_2^2]}{2}
      \operatorname{diag}(1/4,3/4)
  =\operatorname{diag}\left(\frac{17}{384},\frac{17}{128}\right).
\end{equation*}

Set $A_n=\operatorname{diag}(n^{1/2},n^{1/4})$. At
$(x,y)=A_n^{-1}(u,v)$ we have
\begin{equation*}
  t_\pm=n^{-1/4}\left(\frac{n^{-1/4}u}{2}
                     \pm\frac{\sqrt{3}\,v}{2}\right).
\end{equation*}
Locally uniformly in $(u,v)$,
$\sqrt{n}\,t_\pm|t_\pm|\to\pm(3/4)v|v|$ and
$\sqrt{n}\,t_\pm^2\to(3/4)v^2$, while $\sqrt{n}|z|^3\to0$.
Using $e_+-e_-=(0,\sqrt{3})$ and $e_++e_-=(1,0)$ gives
\begin{equation*}
  \sqrt{n}\nabla F(A_n^{-1}(u,v))
  \longrightarrow
  \Lambda(u,v)=\begin{pmatrix}2u/3+\sqrt{3}\,v^2/6\\ v|v|\end{pmatrix}
\end{equation*}
locally uniformly. The second coordinate uniquely determines $v$,
and the first then uniquely determines $u$, with both inversions
continuous. Thus $\Lambda$ is a homeomorphism, and
Theorem~\ref{thm:gradient-scaling-limit} yields
\begin{equation*}
  \left(\sqrt{n}\,x_n,n^{1/4}y_n\right)
  \xrightarrow{d}
  \left(\frac32\zeta_1-\frac{\sqrt{3}}4|\zeta_2|,
         \operatorname{sgn}(\zeta_2)\sqrt{|\zeta_2|}\right).
\end{equation*}
The Gaussian variables $\zeta_1$ and $\zeta_2$ are independent.
The first coordinate retains the classical $\sqrt{n}$ scaling.
The term $-(\sqrt{3}/4)|\zeta_2|$ makes its limit non-Gaussian
and makes the two components dependent.
\end{proof}

\subsection{A construction on the circle}

We now describe a construction of probability distributions on the circle that we will use in the following examples.
Identify $S^1$ with $[-\pi,\pi)$, equipped
with its geodesic distance, and fix $\delta\in(0,\pi/2)$.
Let $h\colon[0,\delta]\to\mathbb{R}$ be absolutely continuous,
with $h(0)=0$, $h(t)>0$ for $t>0$, and
$0\leq h'(t)<1$ for almost every $t\in(0,\delta)$.
Define a probability distribution by
\begin{equation}\label{eq:circle-example-law}
P(d\theta)
=\left(1-\frac{\delta-h(\delta)}{\pi}\right)\delta_0(d\theta)
+p(\theta)\,d\theta,
\qquad
p(\theta)=
\begin{cases}
\dfrac{1-h'(\pi-|\theta|)}{2\pi},
&\pi-\delta<|\theta|<\pi,\\
0,&\text{otherwise}.
\end{cases}
\end{equation}
Set $p=0$ on the null set where $h'$ is undefined.
Then $p\geq0$ and
\begin{equation*}
\int_{-\pi}^{\pi}p(\theta)\,d\theta
=\frac{\delta-h(\delta)}{\pi}\in(0,1).
\end{equation*}
Symmetry gives $\mathbb{E}[X]=0$, where $X$ is viewed as
a real-valued angle. For $0<x<\pi$, the derivative of
half the squared circular distance is
\begin{equation*}
\frac{\partial}{\partial x}\frac12d_{S^1}(x,\theta)^2
=x-\theta-2\pi\mathbf{1}_{[-\pi,-\pi+x)}(\theta)
\end{equation*}
except at the antipodal angle $\theta=x-\pi$, which has
probability zero. The difference quotients are bounded
by $\pi$, so dominated convergence permits differentiation
under the expectation. For $0<x<\delta$,
\begin{equation*}
P([-\pi,-\pi+x))
=\frac{1}{2\pi}\int_0^x(1-h'(t))\,dt
=\frac{x-h(x)}{2\pi};
\end{equation*}
for $\delta\leq x<\pi$, this probability equals
$(\delta-h(\delta))/(2\pi)$. Hence
\begin{equation*}
F'(x)
=x-2\pi P([-\pi,-\pi+x))
=\begin{cases}
h(x),&0<x<\delta,\\
x-\delta+h(\delta),&\delta\leq x<\pi.
\end{cases}
\end{equation*}
The derivative is positive, so symmetry and continuity
show that $0$ is the unique population Fr\'echet mean.
In normal coordinates at $0$,
\begin{equation*}
f'(x)=\operatorname{sgn}(x)h(|x|),\qquad |x|<\delta,
\end{equation*}
where $\operatorname{sgn}(0)=0$.
The function $f'$ is Lipschitz, and $S^1$ is compact.
Thus Theorem~\ref{thm:riemannian-gradient-clt} applies,
with $\sigma^2=\operatorname{Var}(X)>0$.

\subsection{Logarithmic smeariness}
\label{subsec:logarithmic-smeariness}

Logarithmic smeariness on the circle was exhibited in
\cite[Example~2]{HundrieserEltznerHuckemann2026}.
We put a density near the antipode of $0$ and assign the remaining mass to an atom at $0$.
We derive its logarithmic limit using
Theorem~\ref{thm:inverse-gradient-scaling-limit}.

\begin{example}
\label{ex:logarithmic-rate}
Fix $\nu>0$ and use \eqref{eq:circle-example-law} with
\begin{equation*}
h(t)=\begin{cases}
\exp(-t^{-\nu}),&t>0,\\
0,&t=0.
\end{cases}
\end{equation*}
Since $h'(t)\to0$ as $t\downarrow0$, a sufficiently
small $\delta$ satisfies the requirements of that
construction. In particular, $0$ is the unique
population Fr\'echet mean, and
\begin{equation*}
f'(x)=\begin{cases}
\operatorname{sgn}(x)\exp(-|x|^{-\nu}),&x\neq0,\\
0,&x=0,
\end{cases}
\qquad |x|<\delta.
\end{equation*}
The function $f$ is smooth near $0$ and
\begin{equation*}
\sqrt{n} f'(x_n)\xrightarrow{d}Z,
\qquad Z\sim\mathcal{N}(0,\sigma^2),\quad\sigma^2>0.
\end{equation*}
Nevertheless, the sample means have the logarithmic limit
\begin{equation*}
(\log\sqrt{n})^{1/\nu}x_n
\xrightarrow{d}\frac12\delta_{-1}+\frac12\delta_1.
\end{equation*}
\end{example}

\begin{proof}
The derivative $f'$ is strictly increasing on $V=(-\delta,\delta)$.
For $n\geq2$, set $A_n=(\log\sqrt n)^{1/\nu}I$ and take
$D=\mathbb{R}\setminus\{0\}$.
For every compact set $K\subset D$ and all sufficiently large $n$,
\begin{equation*}
A_n(f'|_V)^{-1}\left(\frac{y}{\sqrt n}\right)
=\operatorname{sgn}(y)
\left(1-\frac{\log|y|}{\log\sqrt n}\right)^{-1/\nu}
\longrightarrow\operatorname{sgn}(y)
\end{equation*}
uniformly for $y\in K$, since $\log|y|$ is bounded on $K$.
As $\mathbb{P}(Z\in D)=1$, Theorem~\ref{thm:inverse-gradient-scaling-limit}
gives $A_nx_n\xrightarrow{d}\operatorname{sgn}(Z)$.
The latter random variable has distribution
$(\delta_{-1}+\delta_1)/2$, as claimed.
\end{proof}

This normalization does not meet the assumptions of
Theorem~\ref{thm:gradient-scaling-limit}.
Indeed, with $A_n=(\log\sqrt{n})^{1/\nu}I$ for $n\geq2$, for every
fixed $y>0$ and all sufficiently large $n$,
\begin{equation*}
\sqrt{n} f'(A_n^{-1}y)=n^{(1-y^{-\nu})/2}.
\end{equation*}
This diverges for $y>1$, so there is no finite locally
uniform limit.

\subsection{Failure of a nondegenerate limit}
\label{subsec:oscillatory-limits}

\begin{example}
\label{ex:oscillatory-no-limit}
Choose $a,\varepsilon>0$ such that
\begin{equation*}
0<a-\sqrt{2}\varepsilon,
\qquad a+\sqrt{2}\varepsilon<1,
\end{equation*}
and use \eqref{eq:circle-example-law} with
\begin{equation*}
h(t)=\begin{cases}
t\left(a+\varepsilon\sin\log\frac1t\right),&t>0,\\
0,&t=0.
\end{cases}
\end{equation*}
Then $0$ is the unique population Fr\'echet mean,
$f\in C^{1,1}$ near $0$, and
\begin{equation*}
f'(x)=x\left(a+\varepsilon\sin\log\frac1{|x|}\right),
\qquad 0<|x|<\delta.
\end{equation*}
The gradient satisfies
\begin{equation*}
\sqrt{n} f'(x_n)\xrightarrow{d}Z,
\qquad Z\sim\mathcal{N}(0,\sigma^2),\quad\sigma^2>0.
\end{equation*}
However, $\sqrt{n} x_n$ is tight and has no distributional
limit. More strongly, no deterministic sequence of
linear isomorphisms $A_n\colon\mathbb{R}\to\mathbb{R}$
makes $A_nx_n$ converge to a nondegenerate distribution.
\end{example}

\begin{proof}
\noindent\textbf{Step 1. Verification of the assumptions and tightness.}
We verify the circular construction and show that
$\sqrt{n} x_n=O_{\mathbb{P}}(1)$.

For $t>0$,
\begin{equation*}
h'(t)=a+\varepsilon\sin\log\frac1t
-\varepsilon\cos\log\frac1t,
\end{equation*}
so $0<a-\sqrt{2}\varepsilon\leq h'(t)
\leq a+\sqrt{2}\varepsilon<1$.
The construction and the manifold theorem therefore apply.
Put $W_n=\sqrt{n} f'(x_n)$, taking the localization
radius smaller than $\delta$. Since
\begin{equation*}
(a-\varepsilon)|x_n|\leq|f'(x_n)|
\leq(a+\varepsilon)|x_n|,
\end{equation*}
the convergence $W_n\xrightarrow{d}Z$ proves tightness.
The oscillation of $f'(x)/x$ also shows that $f$ is not
twice differentiable at $0$.

\medskip
\noindent\textbf{Step 2. Distinct subsequential limits.}
We identify a family of limits indexed by a phase.

For $\vartheta\in\mathbb{R}$, define
\begin{equation*}
H_\vartheta(y)=\begin{cases}
y\left(a+\varepsilon
\sin\left(\vartheta+\log\frac1{|y|}\right)\right),&y\neq0,\\
0,&y=0.
\end{cases}
\end{equation*}
Each $H_\vartheta$ is an odd increasing homeomorphism
of $\mathbb{R}$, with derivatives bounded below by
$a-\sqrt{2}\varepsilon>0$ away from $0$.
Moreover,
\begin{equation*}
W_n=H_{\frac12\log n}(\sqrt{n} x_n).
\end{equation*}
Let $\Phi$ denote the standard normal distribution
function. Uniform convergence of the distribution
functions of $W_n$ gives, for every fixed $y\in\mathbb{R}$,
\begin{equation*}
\mathbb{P}(\sqrt{n} x_n\leq y)
=\Phi\left(\frac{H_{\frac12\log n}(y)}{\sigma}\right)+o(1).
\end{equation*}
For $\vartheta\in[0,2\pi)$, take
$n_k=\lfloor\exp(4\pi k+2\vartheta)\rfloor$.
Then
\begin{equation*}
\sqrt{n_k}x_{n_k}\xrightarrow{d}H_\vartheta^{-1}(Z).
\end{equation*}
At $y=1$, the limiting distribution functions for
$\vartheta=0$ and $\vartheta=\pi/2$ are, respectively,
$\Phi(a/\sigma)$ and $\Phi((a+\varepsilon)/\sigma)$.
Thus the full sequence cannot converge in distribution.

\medskip
\noindent\textbf{Step 3. Exclusion of other linear normalizations.}
We show that no deterministic scalar normalization
has a nondegenerate distributional limit.

Write $A_nx=a_nx$, where $a_n\neq0$, and suppose
that $a_nx_n$ converges to a nondegenerate distribution.
Set $c_n=|a_n|/\sqrt{n}$.
No subsequence of $c_n$ can tend to zero, by tightness
of $\sqrt{n} x_n$.
Nor can a subsequence tend to infinity: for any $M>0$,
along such a subsequence,
\begin{equation*}
\mathbb{P}(|a_nx_n|>M)
\geq\mathbb{P}\left(|W_n|>
\frac{(a+\varepsilon)M}{c_n}\right)\longrightarrow1,
\end{equation*}
contradicting tightness. Hence $c_n$ is bounded and
bounded away from zero.

Along further subsequences of the phase-$0$ and
phase-$\pi/2$ sequences from Step~2, let
$c_n\to c_0$ and $c_n\to c_1$, respectively.
Passing to fixed signs of $a_n$ if necessary, symmetry
of $H_\vartheta^{-1}(Z)$ implies
\begin{equation*}
c_0H_0^{-1}(Z)\stackrel d=c_1H_{\pi/2}^{-1}(Z).
\end{equation*}
More generally, if $c,d>0$ and
$cH_\vartheta^{-1}(Z)\stackrel d=dH_\psi^{-1}(Z)$,
strict monotonicity of $\Phi$ gives
$H_\vartheta(y/c)=H_\psi(y/d)$ for every $y$.
For $y>0$, putting $t=\log(1/y)$ yields
\begin{equation*}
\frac1c\left(a+\varepsilon\sin(\vartheta+\log c+t)\right)
=\frac1d\left(a+\varepsilon\sin(\psi+\log d+t)\right),
\qquad t\in\mathbb{R}.
\end{equation*}
Averaging over a period gives $c=d$, and equality of
the sinusoidal terms then gives
$\vartheta=\psi$ modulo $2\pi$.
This is impossible for $0$ and $\pi/2$, proving the claim.
\end{proof}

\begin{acknowledgments}
The author thanks Kenta Misumi for helpful discussions during seminars and Makiko Sasada for valuable comments during seminars. The author also thanks Shin-ichi Ohta for suggesting that the main result could be extended to the case of semiconvex squared distance functions.
\end{acknowledgments}

\section*{Statement on AI use}
The author used OpenAI's generative AI tools to assist with manuscript organization, exposition, and literature searches, and used GPT-5.6 Sol to work through explicit calculations in the examples. The author takes full responsibility for the final content of the manuscript.

\bibliography{CLT_ref}
\end{document}